\documentclass[11pt, reqno]{amsart}
\usepackage[T1]{fontenc}
\usepackage[utf8]{inputenc}
\usepackage[english]{babel}
\usepackage{amsmath,amsfonts,amssymb}
\usepackage{amsthm}
\usepackage{mathrsfs}
\usepackage{accents,mathtools,extarrows}
\usepackage{yhmath}
\usepackage{graphicx}
\usepackage{rotating}
\usepackage{wasysym}
\usepackage{textcomp}
\usepackage{hyperref,cleveref}
\usepackage{xfrac}
\usepackage{enumitem}
\usepackage{tikz-cd}									
\usetikzlibrary{matrix,arrows,calc,positioning,decorations.pathmorphing}
\tikzset{
  shift left/.style ={commutative diagrams/shift left={#1}},
  shift right/.style={commutative diagrams/shift right={#1}}
}
\usepackage{todonotes}

\theoremstyle{plain}
\newtheorem{theorem}{Theorem}[section]
\newtheorem{lemma}[theorem]{Lemma}
\newtheorem{proposition}[theorem]{Proposition}
\newtheorem{corollary}[theorem]{Corollary}

\theoremstyle{definition}
\newtheorem{definition}[theorem]{Definition}
\AtBeginEnvironment{definition}{\pushQED{\qed}}
\AtEndEnvironment{definition}{\popQED\enddefinition}
\newtheorem{example}[theorem]{Example}
\AtBeginEnvironment{example}{\pushQED{\qed}}
\AtEndEnvironment{example}{\popQED\endexample}
\newtheorem{remark}[theorem]{Remark}
\AtBeginEnvironment{remark}{\pushQED{\qed}}
\AtEndEnvironment{remark}{\popQED\endremark}
\newtheorem{construction}[theorem]{Construction}
\AtBeginEnvironment{construction}{\pushQED{\qed}}
\AtEndEnvironment{construction}{\popQED\endconstruction}
\newtheorem{observation}[theorem]{Observation}
\AtBeginEnvironment{observation}{\pushQED{\qed}}
\AtEndEnvironment{observation}{\popQED\endobservation}
\newtheorem{assumption}[theorem]{Assumption}
\AtBeginEnvironment{assumption}{\pushQED{\qed}}
\AtEndEnvironment{assumption}{\popQED\endassumption}

\newtheorem*{thexreftheorem}{Theorem \thexref}

\newtheorem*{thexrefcorollary}{Corollary \thexcoro}
\newcommand{\C}{\mathbb C}

\renewcommand{\epsilon}{\varepsilon}
\renewcommand{\phi}{\varphi}

\begin{document}

\title{Stratified homotopy theory of topological $\infty$-stacks: a toolbox}
\author[Mikala Ørsnes Jansen]{Mikala Ørsnes Jansen}
\address{Department of Mathematical Sciences, University of Copenhagen, 2100 Copenhagen, Denmark.}
\email{mikala@math.ku.dk}
\thanks{The author was supported by the European Research Council (ERC) under the European Unions Horizon 2020 research and innovation programme (grant agreement No 682922) and the Danish National Research Foundation through the Copenhagen Centre for Geometry and Topology (DNRF151).}

\begin{abstract}
We exploit the theory of $\infty$-stacks to provide some basic definitions and calculational tools regarding stratified homotopy theory of stratified topological stacks.
\end{abstract}

\maketitle

\tableofcontents

\section{Introduction}

The purpose of this note is to define the stratified homotopy type of a topological ($\infty$-)stack and establish some basic calculational tools. In order to do so, we must review the basics of $\infty$-stacks and sheaves on these. The results are neither deep nor surprising, but we were in need of a rigorous setup to make the explicit calculation in \cite{OrsnesJansen23a} where we determine the stratified homotopy type of the Deligne--Mumford--Knudsen compactification, that is, the moduli stack of stable nodal curves. We hope the tools may be of use in other settings as well or at the very least that they provide some basic insight into sheaves on stacks.

We begin by recapping and setting up the relevant background: we recall the definitions of $\infty$-sites, $\infty$-stacks and sheaves on $\infty$-stacks and move on to some fundamental descent results (\S \ref{stacks and sheaves}). We define stratified (topological) $\infty$-stacks and constructible sheaves on stratified étaletale  $\infty$-stacks, and we then introduce the notion of exit path $\infty$-categories of stratified étale $\infty$-stacks (\S \ref{stratified homotopy theory}). Among other things, we exhibit two useful permanence properties of exit path $\infty$-categories that provide a powerful tool for calculations (\Cref{permanence properties}). Finally, as a small aside, we consider the procedure of changing the base category of a stack; more specifically, we provide an explicit characterisation of the underlying topological $\infty$-stack $X^{\operatorname{top}}$ of an algebraic $\infty$-stack $X$ (\S \ref{change of base}).

We remark here that the results of this note apply mostly to étale topological $\infty$-stacks, that is, stacks admitting an étale atlas. As a consequence, the present setting has some serious limitations as such stacks are for the most part $1$-truncated. Importantly, however, they include the topological stacks underlying classical algebraic Deligne--Mumford $1$-stacks. See \Cref{n-stacks}, \Cref{restrictions of LH(X)} and \Cref{not Artin stacks}.

\medskip

\textbf{Existing literature.}
The content of this note is just a natural generalisation of the notion of exit path $\infty$-categories of conically stratified topological spaces as developed by Lurie in \cite[Appendix A]{LurieHA}. We opt, however, for a less constructive approach: essentially, the exit path $\infty$-category of a stratified $\infty$-stack will be the idempotent complete $\infty$-category classifying constructible sheaves, if it exists (see also \cite[\S 3]{ClausenOrsnesJansen}).

In \cite{BarwickHaine19a}, Barwick and Haine build on the work of \cite{BarwickGlasmanHaine} and set up a stratified étale homotopy theory of algebraic stacks, but this is a very different setting from the one in this note. Also, our ambitions with this note are more modest: we simply needed an easily manageable setup in which to make our calculations. Thus we exploit the full strength of $\infty$-stacks to make some simple observations and definitions introducing only a minimum of assumptions. Even though we are ultimately interested in ``ordinary'' topological stacks, that is, $1$-truncated $\infty$-stacks, viewing these as $\infty$-stacks makes the arguments run a lot neater and we would in any case need the full strength of studying sheaves valued in the $\infty$-category $\mathcal{S}$ of anima.

In \cite[Appendices A and B]{Nocera}, Nocera also treats stratified homotopy theory of stratified topological stacks, but of a slightly different flavour. One key difference is that while Nocera works with the geometrically defined notion of conicality, we trade this for purely categorical conditions on the $\infty$-category of constructible sheaves.

After this note was first written, Haine--Porta--Teyssier have further developed this unconstructive approach of studying stratified spaces via concrete conditions on the $\infty$-category of constructible sheaves (\cite{HainePortaTeyssier24}). They provide a beautiful unified theory that incorporates a large class of natural and interesting examples of stratified spaces that were not approachable with the previously available methods. Previous to this work, Lurie's exodromy equivalence has been extended by Lejay, who systematically uses hypersheaves to relax the finiteness conditions of the stratifying poset (\cite{Lejay}). Porta--Teyssier have also studied Lurie's exodromy equivalence in a purely $\infty$-categorical language, revealing among other things that it is functorial with respect to arbitrary stratified maps (\cite{PortaTeyssier}).

Let us also briefly mention that as a byproduct of our setup, we define the homotopy type of a topological $\infty$-stack (when it admits one, see \Cref{definition exit path category}). There are already various papers dealing with the homotopy theory of stacks and the basics of topological stacks, for example \cite{Noohi, Ebert, Carchedi}. Our approach is less constructive, our ultimate criterion being a classification of locally constant (or more generally constructible) sheaves, but this also means that we can avoid introducing technical assumptions on the stacks under consideration and our definition does not require us to choose an atlas.

\medskip

\textbf{Acknowledgements.} 
The author would like to thank Dustin Clausen for many fruitful conversations, for the (also many) comments and remarks on drafts of this note and not least for the collaboration in \cite{ClausenOrsnesJansen}; the definitions, results and proofs regarding exit path $\infty$-categories presented here are direct generalisations of the work done in that paper. We would also like to thank Søren Galatius for valuable comments and insights. Finally, we thank two anonymous referees for their valuable comments and pointers.

\medskip

\textbf{Notation and conventions} We adopt the set-theoretic conventions of \cite[\S 1.2.15]{LurieHTT}.

\section{Stacks and sheaves}\label{stacks and sheaves}

In this section, we recall the definitions of sheaves on $\infty$-sites, $\infty$-stacks, and sheaves on $\infty$-stacks.

\subsection{Sheaves on sites}

By an \textit{$\infty$-site}, we mean a small $\infty$-category $C$ equipped with a Grothendieck topology, i.e.~a specification of covering sieves for each object in $C$ satisfying certain axioms (see \cite[Definition 6.2.2.1]{LurieHTT}).

Let $\mathcal{E}$ be an $\infty$-category and consider the $\infty$-category of $\mathcal{E}$-valued presheaves $\mathcal{P}_{\mathcal{E}}(C)=\operatorname{Fun}(C^{\operatorname{op}},\mathcal{E})$. A presheaf $F\colon C^{\operatorname{op}}\rightarrow \mathcal{E}$ is a \textit{sheaf} if for any object $x\in C$ and any covering sieve $C_{/x}^0\subseteq C_{/x}$, the composite
\begin{align*}
(C_{/x}^0)^{\triangleleft}\hookrightarrow C_{/x}^{\triangleleft}\rightarrow C\xrightarrow{F^{\operatorname{op}}} \mathcal{E}^{\operatorname{op}}
\end{align*}
is a colimit diagram in $\mathcal{E}^{\operatorname{op}}$. We denote by $\operatorname{Shv}(C; \mathcal{E})\subseteq \mathcal{P}_{\mathcal{E}}(C)$ the full subcategory spanned by the $\mathcal{E}$-valued sheaves on $C$.

For $\mathcal{E}=\mathcal{S}$, the $\infty$-category of spaces, we write $\mathcal{P}(C):=\mathcal{P}_{\mathcal{S}}(C)$ and $\operatorname{Shv}(C):=\operatorname{Shv}(C;\mathcal{S})$, and we recall that $\operatorname{Shv}(C)$ can be identified as a localisation of $\mathcal{P}(C)$. In view of \Cref{sheaves with general coefficients from space valued sheaves} below, it will generally suffice for us to deal with $\mathcal{S}$-valued sheaves, so let us spell out this localisation in slightly more detail. For a given $\infty$-site $C$, consider the Yoneda embedding $y\colon C\rightarrow \mathcal{P}(C)$ and let $S$ denote the collection of monomorphisms $u\rightarrow y(x)$ corresponding to covering sieves on $x$ for all objects $x$ in $C$ (\cite[Proposition 6.2.2.5]{LurieHTT}). The category of sheaves on $C$ is the full subcategory $\operatorname{Shv}(C)\subseteq \mathcal{P}(C)$ of $S$-local presheaves on $C$ (\cite[Definition 6.2.2.6]{LurieHTT}). This identifies $\operatorname{Shv}(C)$ as a topological localisation of $\mathcal{P}(C)$ and thus in particular an $\infty$-topos (\cite[Proposition 6.2.2.7]{LurieHTT}).

We refer to \cite[\S 6.2.2]{LurieHTT} for details. See also \cite[\S\S 2.2-2.3]{PortaYu16} where these observations are worked out carefully.

\begin{remark}\label{sheaves with general coefficients from space valued sheaves}
Let $\mathcal{E}$ be a presentable $\infty$-category. The $\infty$-categories $\mathcal{P}_{\mathcal{E}}(C)$ and $\operatorname{Shv}(C; \mathcal{E})$ can be naturally identified with the tensor products $\mathcal{P}(C)\otimes \mathcal{E}$, respectively $\operatorname{Shv}(C)\otimes \mathcal{E}$ (\cite[Remark 1.3.1.6 and Proposition 1.3.1.7]{LurieSAG}, see also \cite[\S 4.8.1]{LurieHTT} for details on tensor products). For our purposes it suffices to consider sheaves valued in compactly generated $\infty$-categories, where the situation is even simpler: if $\mathcal{E}$ is a compactly generated $\infty$-category and $\mathcal{E}_0$ denotes the full subcategory of compact objects, then 
\begin{align*}
\operatorname{Shv}(C;\mathcal{E})\xrightarrow{\ \sim \ } \operatorname{Fun}^{\operatorname{lex}}(\mathcal{E}_0^{\operatorname{op}},\operatorname{Shv}(C)),
\end{align*}
where the right hand side denotes the full subcategory spanned by the functors preserving finite limits (see e.g.~\cite[Appendix B]{OrsnesJansen} where the relevant observations are made for the site $\mathscr{U}(X)$ of open subsets of a topological space $X$).
\end{remark}

We will need the following lemma where $y\colon C\rightarrow \mathcal{P}(C)$ denotes the Yoneda embedding as above and $L\colon \mathcal{P}(C)\rightarrow \operatorname{Shv}(C)$ the sheafification functor. Recall that for an object $x$ in $C$, the \textit{big site} $C_{/x}$ of $x$ is the over-$\infty$-category $C_{/x}$ equipped with the following Grothendieck topology: the covering families are collections of morphisms whose image under the functor $C_{/x}\rightarrow C$ are covering families in $C$.

\begin{lemma}\label{over categories and sheaves}
Let $C$ be a small $\infty$-site, let $x$ be an object of $C$ and consider the big site $C_{/x}$ of $x$. The left adjoint to the restriction along $C_{/x}\rightarrow C$ induces an equivalence
\begin{align*}
\operatorname{Shv}(C_{/x})\xrightarrow{\sim} \operatorname{Shv}(C)_{/Ly(x)}.
\end{align*}
Moreover, for any morphism of $\infty$-sites $f\colon C\rightarrow D$, we have a homotopy commutative diagram where the vertical maps are given by the left adjoint to restriction along $f$:
\begin{center}
\begin{tikzpicture}
\matrix (m) [matrix of math nodes,row sep=2em,column sep=2em]
  {
\operatorname{Shv}(C_{/x}) & \operatorname{Shv}(C)_{/Ly(x)} \\
\operatorname{Shv}(D_{/f(x)}) & \operatorname{Shv}(D)_{/Ly(f(x))} \\
  };
  \path[-stealth]
(m-1-1) edge node[above]{$\sim$} (m-1-2)
(m-2-1) edge node[below]{$\sim$} (m-2-2)
(m-1-2) edge (m-2-2)
(m-1-1) edge (m-2-1)
  ;
\end{tikzpicture}
\end{center}
\end{lemma}
\begin{proof}
By \cite[Proposition 5.2.5.1]{LurieHTT}, the sheafification adjunction $(L\dashv \iota)\colon \mathcal{P}(C)\rightleftarrows \operatorname{Shv}(C)$ passes to an adjunction
\begin{align*}
(L'\dashv \iota')\colon \mathcal{P}(C)_{/y(x)}\rightleftarrows \operatorname{Shv}(C)_{/Ly(x)}.
\end{align*}
The resulting right adjoint $\iota'$ is fully faithful since the counit is given by the counit of the adjunction $L\dashv \iota$ which is an equivalence, $\iota$ being fully faithful. It follows that this exhibits $\operatorname{Shv}(C)_{/Ly(x)}$ as a reflective subcategory of $\mathcal{P}(C)_{/y(x)}$.

By the universal property of presheaf categories, we have a colimit preserving functor
\begin{align*}
\mathcal{P}(C_{/x})\rightarrow \mathcal{P}(C)_{/y(x)}
\end{align*}
which sends a representable $t\colon c\rightarrow x$ to $y(t)\colon y(c)\rightarrow y(x)$; in other words, when composed with the projection to $\mathcal{P}(C)$ it agrees with the left adjoint to the restriction along $C_{/x}\rightarrow C$ (\cite[Corollary F]{HaugsengHebestreitLinskensNuiten}). By \cite[Corollary 5.1.6.12]{LurieHTT}, this functor is an equivalence. Thus is suffices to note that the localisation
\begin{align*}
\mathcal{P}(C)_{/y(x)}\xrightarrow{L'} \operatorname{Shv}(C)_{/Ly(x)}
\end{align*}
precisely inverts those morphisms whose image under the projection to $\mathcal{P}(C)$ are inverted by $L$, and that this exactly recovers the covering monomorphisms in $\mathcal{P}(C_{/x})$ corresponding to the big site structure on $C_{/x}$. The final claim follows by uniqueness of left adjoints and the fact that the left adjoint to restriction along $f_{/x}$ sends a representable $t\colon c\rightarrow x$ to the representable given by its value under $f_{/x}$, i.e.~$f(t)\colon f(c)\rightarrow f(x)$.
\end{proof}

We shall also need the following observation; recall that a morphism $f\colon u\rightarrow x$ in an $\infty$-topos is an \textit{effective epimorphism} if its \v{C}ech nerve is a simplicial resolution of $x$ (see \cite[Corollary 6.2.3.5]{LurieHTT}).

\begin{lemma}\label{yoneda maps atlas to effective epi}
Let $C$ be an $\infty$-site and consider the composite $C\xrightarrow{y} \mathcal{P}(C)\xrightarrow{L}\operatorname{Shv}(C)$ of the Yoneda embedding and the sheafification functor. Let $\alpha\colon u \rightarrow x$ be a morphism in $C$. If the sieve on $x$ generated by $\alpha$ is a covering in $C$, then $Ly(\alpha)\colon Ly(u)\rightarrow Ly(x)$ is an effective epimorphism.
\end{lemma}
\begin{proof}
Let $\check{C}(y(\alpha))$ denote the \v{C}ech nerve of $y(\alpha)$ and consider its $(-1)$-truncation, the monomorphism $\phi\colon \varinjlim \check{C}(y(\alpha))\rightarrow y(x)$ (\cite[Proposition 6.2.3.4]{LurieHTT}). By \cite[Lemma 6.2.2.16]{LurieHTT}, $Ly(\alpha)$ is an effective epimorphism if and only if $\phi$ is a covering monomorphism in $\mathcal{P}(C)$, since $L$ is a left exact left adjoint and thus preserves both \v{C}ech nerves and colimits. Now, by definition, $\phi$ is a covering monomorphism if and only if the corresponding sieve $S_\phi$ is a covering sieve. The sieve $S_\phi$ consists of $f\colon c\rightarrow x$ such that $y(f)$ factors through $\phi$ (\cite[Proposition 6.2.2.5]{LurieHTT}). Obviously, $\alpha\in S_\phi$, and thus the sieve $S_\alpha$ generated by $\alpha$ is contained in $S_\phi$. Hence, if $S_\alpha$ is a covering, so is $S_\phi$ and this concludes the proof.
\end{proof}

\subsection{\texorpdfstring{$\infty$}{infinity}-stacks and sheaves on \texorpdfstring{$\infty$}{infinity}-stacks}

Consider the category $\mathfrak{T}=\operatorname{LCHaus}^{2nd}\subset \operatorname{Top}$ of second-countable locally compact Hausdorff topological spaces. We equip $\mathfrak{T}$ with the Grothendieck topology generated by open covers, and consider the $\infty$-category $\operatorname{Shv}(\mathfrak{T})$ of \textit{(second countable locally compact Hausdorff) topological $\infty$-stacks}, that is, space-valued sheaves on $\mathfrak{T}$. The category $\mathfrak{T}$ embeds fully faithfully into $\operatorname{Shv}(\mathfrak{T})$ via the Yoneda embedding. In fact, any topological space $X\in \operatorname{Top}$ defines an $\infty$-stack $\underline{X}\colon Y\mapsto \operatorname{Map}(Y,X)$. By slight abuse of notation, we will also denote the $\infty$-stack $\underline{X}$ by $X$. Unless specifying otherwise, we assume from now on that all our topological spaces belong to $\mathfrak{T}$, but note that we will be encountering other topological spaces as we will be considering posets as topological spaces via the Alexandroff topology (\Cref{stratified homotopy theory}).

\begin{remark}
Restricting our attention to the subcategory $\mathfrak{T}$ ensures that we do not run into set-theoretic difficulties --- we could consider any sufficiently ``small'' category of topological spaces as long as it admits fibre products. Note, however, that we do explicitly use the locally compact Hausdorff assumption for the proper base change theorem (\Cref{proper base change}, see also \Cref{remark proper base-change}).
\end{remark}

We say that a morphism $Y\rightarrow X$ of $\infty$-stacks is \textit{representable} if for any map $S\rightarrow X$ with $S\in \mathfrak{T}$, the pullback $S\times_X Y$ also belongs to $\mathfrak{T}$. If a property $\mathbf{P}$ of maps of topological spaces is invariant under base change, then we say that a representable map $Y\rightarrow X$ of $\infty$-stacks has the property $\mathbf{P}$, if its pullback $S\times_X Y\rightarrow S$ has the property $\mathbf{P}$ for any map $S\rightarrow X$ with $S\in \mathfrak{T}$.

By an \textit{atlas} of an $\infty$-stack $X\in \operatorname{Shv}(\mathfrak{T})$, we mean a topological space $U\in \mathfrak{T}$ equipped with a representable effective epimorphism $f\colon U\rightarrow X$ in $\operatorname{Shv}(\mathfrak{T})$. Consequently, $X$ can be realised as the colimit of the \v{C}ech nerve $\check{C}(f)$ of $f$. We say that an atlas is \textit{étale} if the morphism $U\rightarrow X$ is a local homeomorphism, i.e.~pulling back along a morphism $S\rightarrow X$ with $S\in \mathfrak{T}$ yields a local homeomorphism $S\times_X U\rightarrow S$ of topological spaces. We say that an $\infty$-stack is \textit{étale} if it admits an étale atlas. Note that we assume $U\in \mathfrak{T}$; in particular, we do not allow $U$ to be an uncountable disjoint union. 

Recall that a property $\mathbf{P}$ of topological spaces that is invariant under base change is said to be \textit{local on the target} if the following holds: a map of topological spaces $f\colon Y\rightarrow X$ has the property $\mathbf{P}$ if and only if there is a cover $\mathcal{U}$ of $X$ such that the pullback $Y\times_X U\rightarrow U$ has the property $\mathbf{P}$ for each $U\in \mathcal{U}$. If a property $\mathbf{P}$ of maps of topological spaces is invariant under base change and moreover local on the target, then we say that a map $Y\rightarrow X$ of étale $\infty$-stacks (not necessarily representable) has the property $\mathbf{P}$, if its pullback $U\times_X Y\rightarrow U$ has the property $\mathbf{P}$ for some (hence any) étale atlas $U\rightarrow X$. For our purposes, it suffices to note that the following notions are invariant under base change and local on the target:
\begin{itemize}
\item local homeomorphisms,
\item open (resp., closed) embeddings,
\item proper maps (as our topological spaces are assumed to be Hausdorff).
\end{itemize}

\begin{remark}
The usual $2$-category of topological stacks (stacks of groupoids on $\mathfrak{T}$) can be identified with the full subcategory of $\operatorname{Shv}(\mathfrak{T})$ spanned by those sheaves $X$ for which $X(S)$ is $1$-truncated for all $S\in \mathfrak{T}$ (see also \cite[Example 3.6.3]{LurieTanaka}). Note that we still restrict our attention to second countable locally compact Hausdorff topological spaces. We recall the useful fact that a morphism $U\rightarrow X$ in $\operatorname{Shv}(\mathfrak{T})$ is an effective epimorphism if and only if its $0$-truncation is an effective epimorphism in the ordinary topos $h(\tau_{\leq 0}\operatorname{Shv}(\mathfrak{T}))$ (\cite[Proposition 7.2.1.14]{LurieHTT}). Hence, if for example $U\rightarrow X$ is an atlas in the $2$-category of $1$-truncated stacks, then $U\rightarrow X$ is also an atlas in the $\infty$-category of $\infty$-stacks.
\end{remark}

\begin{observation}\label{n-stacks}
As a converse to the above remark, we note that if an $\infty$-stack $X\in \operatorname{Shv}(\mathfrak{T})$ admits an étale atlas and the diagonal $X\rightarrow X\times X$ is representable, then $X$ is in fact $1$-truncated: indeed, for any $S\in \mathfrak{T}$, we may identify
\begin{align*}
\pi_i(X(S),\alpha)\cong \pi_{i-1}((S\times_X S)(S), (\operatorname{id}, \operatorname{id}))
\end{align*}
where the pullback is taken along the map $S\rightarrow X$ corresponding to the point $\alpha\in X(S)$ and $\operatorname{id}$ denotes the point in $S(S)$ given by the identity map. Noting that
\begin{align*}
S\mathop{\times}_X S\simeq X\mathop{\times}_{X\times X}(S\times S),
\end{align*}
we see that this is representable and, in particular, $0$-truncated.

This observation reveals certain limitations of the present note as we focus on stacks admitting a representable étale atlas. A more complete approach would consider higher geometric stacks in the sense of Simpson, that is, inductively defined $n$-geometric stacks as in e.g. \cite[\S 1.3.3]{ToenVezzosi}. An $\infty$-stack $X$ with an étale atlas and whose diagonal is representable is a $0$-geometric stack in this setting. We content ourselves with the present approach as the existence of an étale atlas serves as a simplifying assumption under which our definitions and observations disentangle considerably; it captures classically arising stacks, providing a tool for studying these in an $\infty$-categorical setting.
\end{observation}

\begin{definition}\label{coarse moduli space}
Let $X\in \operatorname{Shv}(\mathfrak{T})$ be an $\infty$-stack. A \textit{coarse moduli space} for $X$ is a topological space $S\in \operatorname{Top}$ together with a morphism $\tau\colon X\rightarrow S$ such that
\begin{enumerate}
\item $\tau$ induces a bijection on points
\begin{align*}
\pi_0 X(\ast)\xrightarrow{\ \cong \ }\operatorname{Map}_{\operatorname{Top}}(\ast,S).
\end{align*}
\item for any topological space $S'\in \operatorname{Top}$ and any morphism $f\colon X\rightarrow S'$, there is an essentially unique factorisation of $f$ through $\tau$:
\begin{center}
\begin{tikzpicture}
\matrix (m) [matrix of math nodes,row sep=2em,column sep=2em]
  {
X & S \\
  & S' \\
  };
  \path[-stealth]
(m-1-1) edge node[above]{$\tau$} (m-1-2)
(m-1-1) edge node[below left]{$f$} (m-2-2)
;
\path[-stealth, dashed]
(m-1-2) edge (m-2-2)
;
\end{tikzpicture}
\end{center}
\end{enumerate} 
We also say that the morphism $\tau$ exhibits $S$ a \textit{coarse moduli space} for $X$. Note that we do not require $S$ to belong to $\mathfrak{T}$.
\end{definition}

We now move on to considering sheaves on $\infty$-stacks. More precisely, we consider the formal extension of the functor $S\mapsto \operatorname{Shv}(S)$ sending a topological space $S\in \mathfrak{T}$ to the $\infty$-category of space-valued sheaves on $S$:
\begin{align*}
\operatorname{Shv}(\mathfrak{T}) \rightarrow \operatorname{Cat}_\infty^{\operatorname{op}},\quad\quad X\mapsto \operatorname{Shv}(X).
\end{align*}
To this end, consider the commutative diagram below where our desired functor is the top horizontal one. The lower horizontal map is the Yoneda embedding, the left vertical map is the sheafification functor, and the right vertical map is the usual sheaf functor on topological spaces with pullback functoriality, i.e.~sheaves on a given topological space $S$ are the space-valued sheaves on the site of open subsets of $S$.
\begin{center}
\begin{equation}\label{sheaf functor diagram}
\begin{aligned}
\begin{tikzpicture}
\matrix (m) [matrix of math nodes,row sep=2em,column sep=2em]
  {
\operatorname{Shv}(\mathfrak{T}) & \operatorname{Cat}_\infty^{\operatorname{op}} \\
\mathcal{P}(\mathfrak{T}) & \mathfrak{T} \\
  };
  \path[-stealth]
(m-2-1) edge node[left]{$L$} (m-1-1)
(m-2-2) edge node[below]{$y$} (m-2-1) edge node[right]{$\operatorname{Shv}$} (m-1-2)
  ;
\path[dashed,-stealth]
(m-1-1) edge (m-1-2)
(m-2-1) edge (m-1-2)
;
\end{tikzpicture}
\end{aligned}
\end{equation}
\end{center}

By the universal property of presheaf categories, the right vertical sheaf functor on $\mathfrak{T}$ determines a colimit preserving functor out of the presheaf category $\mathcal{P}(\mathfrak{T})$ (the diagonal arrow). This functor in turn is readily verified to send covering monomorphisms to equivalences, and thus it factors through the sheafification functor giving rise to the upper horizontal map.

\begin{definition}
Let $X\in \operatorname{Shv}(\mathfrak{T})$ be an $\infty$-stack. The $\infty$-category $\operatorname{Shv}(X)$ of \textit{(space-valued) sheaves} on $X$ is the value of $X$ under the top horizontal functor in the diagram (\ref{sheaf functor diagram}).  
\end{definition}

\begin{remark}\label{sheaves via cech nerve}
It follows immediately from the definition that if $f\colon U\rightarrow X$ is an atlas for a stack $X\in \operatorname{Shv}(\mathfrak{T})$, then $\operatorname{Shv}(X)$ can be realised as a limit in $\operatorname{Cat}_\infty$, $\operatorname{Shv}(X)\simeq\varprojlim \operatorname{Shv}\circ \check{C}(f)$, given by composing the sheaf functor with the \v{C}ech nerve of $f$ on which sheaves are just sheaves on topological spaces.
\end{remark}

We now consider an a priori different notion of sheaves on a stack. Let $X\in \operatorname{Shv}(\mathfrak{T})$ be an $\infty$-stack, and consider the full subcategory $\mathcal{LH}(X)\subset \operatorname{Shv}(\mathfrak{T})_{/X}$ consisting of representable local homeomorphisms $Y\rightarrow X$ with $Y\in \mathfrak{T}$. We equip $\operatorname{Shv}(\mathfrak{T})$ with the canonical topology as defined in \cite[\S 6.2.4]{LurieHTT}, i.e.~a sieve on an object $C$ is a covering if it contains a collection of morphisms $C_i\rightarrow C$, $i\in I$, such that the induced map $\coprod_{i\in I} C_i\rightarrow C$ is an effective epimorphism. We give $\operatorname{Shv}(\mathfrak{T})_{/X}$ the subsequent big site topology and consider $\mathcal{LH}(X)$ as a subsite. We call $\mathcal{LH}(X)$ the \textit{étale site} of $X$. Note that $\mathcal{LH}(X)$ is in fact a $1$-category as the Yoneda embedding is fully faithful (see however \Cref{restrictions of LH(X)} below).

\begin{definition}
Let $X\in \operatorname{Shv}(\mathfrak{T})$. The $\infty$-category of \textit{(space-valued) étale sheaves} on $X$ is the $\infty$-category of sheaves on the étale site of $X$, $\operatorname{Shv}_{\acute{e}t}(X):=\operatorname{Shv}(\mathcal{LH}(X))$.  
\end{definition}

For a morphism of $\infty$-stacks $f\colon Y\rightarrow X$, pullback along $f$ defines a morphims of sites
\begin{align*}
\mathcal{LH}(X)\rightarrow \mathcal{LH}(Y)
\end{align*}
and via restriction a morphism of sheaf categories
\begin{align*}
f_*\colon \operatorname{Shv}_{\acute{e}t}(Y)\rightarrow \operatorname{Shv}_{\acute{e}t}(X).
\end{align*}

If $f$ is étale representable, then postcomposition by $f$ defines a left adjoint to the morphism given by pullback along $f$:
\begin{align*}
\mathcal{LH}(X)\leftrightarrows \mathcal{LH}(Y).
\end{align*}
We denote the arising adjunction of sheaf categories by
\begin{align*}
f^*\vdash f_*\colon \operatorname{Shv}_{\acute{e}t}(X)\leftrightarrows \operatorname{Shv}_{\acute{e}t}(Y).
\end{align*}

Not surprisingly, considering étale sheaves on an $\infty$-stack is only a reasonable thing to do if $X$ is suitably ``geometric''. If, however, $X$ admits an étale atlas, then étale sheaves recover the previous definition, thus providing a simpler way of studying sheaves on $X$ as they are just sheaves on the site $\mathcal{LH}(X)$. We prove this below, but first of all, we make a small remark and prove a lemma that will come in handy again later on.

\begin{remark}\label{restrictions of LH(X)}
Let us remark that our definition of $\mathcal{LH}(X)$ is slightly restrictive as we ask of the maps $Y\rightarrow X$ that $Y$ is representable by an actual topological space instead of just a topological stack of lesser complexity (see also \Cref{n-stacks} above). This fits our purpose, capturing classical $1$-stacks, so for the sake of simplicity and accessibility, we refrain from treating the more general setting here.
\end{remark}

\begin{lemma}\label{sheaves on open substack generalised}
For any representable local homeomorphism $\alpha\colon Y\rightarrow X$ of $\infty$-stacks with $Y\in \mathfrak{T}$, the left adjoint to restriction along $\alpha$ induces an equivalence $\alpha_!\colon \operatorname{Shv}_{\acute{e}t}(Y)\xrightarrow{\sim} \operatorname{Shv}_{\acute{e}t}(X)_{/y(\alpha)}$. Moreover, this description is compatible with basechange: for a pullback square of $\infty$-stacks as on the left below with $\alpha$ (and hence $\alpha'$) a representable local homeomorphism and $Y\in \mathfrak{T}$ (and hence $Y'\in \mathfrak{T}$), the induced diagram on the right is homotopy commutative.
\begin{center}
\begin{tikzpicture}
\matrix (m) [matrix of math nodes,row sep=2em,column sep=2em]
  {
Y' & X' & & \operatorname{Shv}_{\acute{e}t}(Y') & \operatorname{Shv}_{\acute{e}t}(X')_{/y(\alpha')}\\
Y & X & & \operatorname{Shv}_{\acute{e}t}(Y) & \operatorname{Shv}_{\acute{e}t}(X)_{/y(\alpha)}\\
  };
  \path[-stealth]
(m-1-1) edge node[left]{$f'$} (m-2-1) edge node[above]{$\alpha'$} (m-1-2)
(m-1-2) edge node[right]{$f$} (m-2-2) 
(m-2-1) edge node[below]{$\alpha$} (m-2-2)
(m-1-4) edge node[above]{$\alpha'_!$} (m-1-5)
(m-2-4) edge node[left]{${f'}^*$}(m-1-4)
(m-2-4) edge node[below]{$\alpha_!$} (m-2-5)
(m-2-5.120) edge node[right]{$f^*$} (m-1-5.240)

;
\end{tikzpicture}
\end{center}
\end{lemma}
\begin{proof}
This follows immediately from \Cref{over categories and sheaves} and the observation that post-composition by $\alpha$ defines an equivalence $\mathcal{LH}(Y)\xrightarrow{\sim}\mathcal{LH}(X)_{/\alpha}$ of $\infty$-sites, where the latter is equipped with the big site structure. Note that essential surjectivity follows from the fact that étale representable maps satisfy the two-out-of-three property.
\end{proof}

For any $\infty$-stack $X\in \operatorname{Shv}(\mathfrak{T})$, we have a comparison functor
\begin{align*}
\Psi_X\colon\operatorname{Shv}(X)\rightarrow \operatorname{Shv}_{\acute{e}t}(X)
\end{align*}
induced by the left adjoints $\Psi_S\colon \operatorname{Shv}(S)\rightarrow \operatorname{Shv}_{\acute{e}t}(S)$ to restriction along the inclusion of sites $\mathcal{U}(S)\hookrightarrow \mathcal{LH}(S)$ for any $S\in \mathfrak{T}$.

\begin{proposition}\label{sheaves vs etale sheaves}
If $X\in \operatorname{Shv}(\mathfrak{T})$ admits an étale atlas, then the functor $\Psi_X\colon\operatorname{Shv}(X)\rightarrow \operatorname{Shv}_{\acute{e}t}(X)$ is an equivalence.
\end{proposition}
\begin{proof}
Observe first of all that if $X\in \mathfrak{T}$, then the statement is immediate since the comparison map is given by the left adjoint to restriction along the inclusion $\mathcal{U}(X)\hookrightarrow \mathcal{LH}(X)$ and any local homeomorphism over $X$ can be refined to an open cover. For the general case, let $f\colon U\rightarrow X$ be an étale atlas and let $U_\bullet:=\check{C}(f)\colon \Delta^{\operatorname{op}}\rightarrow \operatorname{Shv}(\mathfrak{T})$ denote the \v{C}ech nerve of $f$. The comparison functor for $X$ then factors as in the diagram below where we see that the lower horizontal map is an equivalence by the observation above and the left vertical map is an equivalence as $\operatorname{Shv}(-)$ preserves colimits.
\begin{center}
\begin{tikzpicture}
\matrix (m) [matrix of math nodes,row sep=2em,column sep=2em]
  {
\operatorname{Shv}(X) & & \operatorname{Shv}_{\acute{e}t}(X) \\
\varprojlim\operatorname{Shv}(U_\bullet) & & \varprojlim\operatorname{Shv}_{\acute{e}t}(U_\bullet) \\
  };
  \path[-stealth]
(m-1-1) edge (m-2-1) edge node[above]{$\Psi_X$} (m-1-3)
(m-1-3) edge (m-2-3) 
(m-2-1) edge node[below]{$\varprojlim \Psi_{U_\bullet}$} (m-2-3)
;
\end{tikzpicture}
\end{center}

Thus it suffices to see that the right vertical map is an equivalence. For any $[n]\in \Delta$, we have a homotopy commutative diagram as below where the left vertical map is given by restriction, the top horizontal map is the canonical equivalence to the comma category over the terminal sheaf $\tau$, the lower horizontal map is the equivalence of \Cref{sheaves on open substack generalised} and finally, the right-hand triangle is the diagram of comma categories induced by the obvious underlying maps of objects (here we are considering pullback functors, see e.g.~\cite[\S 6.1.1]{LurieHTT}).
\begin{center}
\begin{tikzpicture}
\matrix (m) [matrix of math nodes,row sep=2em,column sep=2em]
  {
\operatorname{Shv}_{\acute{e}t}(X) & \operatorname{Shv}_{\acute{e}t}(X)_{/\tau} \\
\operatorname{Shv}_{\acute{e}t}(U_n) & \operatorname{Shv}_{\acute{e}t}(X)_{/y(U_n\rightarrow X)} & \operatorname{Shv}_{\acute{e}t}(X)_{/\varinjlim y(U_\bullet\rightarrow X)} \\
  };or
  \path[-stealth]
(m-1-1) edge node[above]{$\sim$} (m-1-2) edge (m-2-1)
(m-1-2) edge (m-2-2) edge (m-2-3)
(m-2-1) edge node[below]{$\sim$} (m-2-2)
(m-2-3) edge (m-2-2)
;
\end{tikzpicture}
\end{center}

Commutativity of the square follows by uniqueness of right adjoints as we recall that the pullback functor of comma categories is right adjoint to the functor of comma categories given by postcomposition with the specified map (\cite[Lemma 6.1.1.1]{LurieHTT}). As the equivalence of \Cref{sheaves on open substack generalised} is compatible with basechange, it follows that we can factorise the map $\operatorname{Shv}_{\acute{e}t}(X)\rightarrow \varprojlim\operatorname{Shv}_{\acute{e}t}(U_\bullet)$ as follows
\begin{center}
\begin{tikzpicture}
\matrix (m) [matrix of math nodes,row sep=2em,column sep=2em]
  {
\operatorname{Shv}_{\acute{e}t}(X) & \varprojlim\operatorname{Shv}_{\acute{e}t}(U_\bullet) \\
\operatorname{Shv}_{\acute{e}t}(X)_{/\varinjlim y(U_\bullet\rightarrow X)} & \varprojlim\operatorname{Shv}_{\acute{e}t}(X)_{/y(U_\bullet\rightarrow X)} \\
  };
  \path[-stealth]
(m-1-1) edge (m-2-1) edge (m-1-2)
(m-1-2) edge node[right]{$\sim$}(m-2-2) 
(m-2-1) edge (m-2-2)
;
\end{tikzpicture}
\end{center}
The lower horizontal map is an equivalence by descent (\cite[\S 6.1.3, more specifically Theorem 6.1.3.9]{LurieHTT}). Now, we claim that the map $\varinjlim y(U_\bullet\rightarrow X)\rightarrow \tau$ is in fact an equivalence; this will finish the proof and is a consequence of $f$ being étale (which we note that we have not actually used yet). To prove the claim, consider the unique morphism $y(f)\rightarrow \tau$. By abuse of notation we will consider this as a morphism in both $\operatorname{Shv}(\operatorname{Shv}(\mathfrak{T})_{/X})$ and $\operatorname{Shv}_{\acute{e}t}(X)$. By \Cref{yoneda maps atlas to effective epi}, $y(f)\rightarrow \tau$ is an effective epimorphism in $\operatorname{Shv}(\operatorname{Shv}(\mathfrak{T})_{/X})$ as it is the image of $f\rightarrow \operatorname{id}_X$ which generates a covering sieve on $\operatorname{id}_X$, $f$ being an effective epimorphism. Since the restriction functor $\operatorname{Shv}(\operatorname{Shv}(\mathfrak{T})_{/X})\rightarrow \operatorname{Shv}_{\acute{e}t}(X)$ is a left exact left adjoint, it preserves effective epimorphisms and \v{C}ech nerves (\cite[Remark 6.2.3.6]{LurieHTT}), and hence $y(f)\rightarrow \tau$ is an effective epimorphism in $\operatorname{Shv}_{\acute{e}t}(X)$ whose \v{C}ech nerves identifies with $y(U_\bullet\rightarrow X)$, and we are done.
\end{proof}

In view of the result above, we will from now on not distinguish between $\operatorname{Shv}_{\operatorname{\acute{e}t}}(X)$ and $\operatorname{Shv}(X)$ for étale $\infty$-stacks. Moreover, it leads us to make the following definition.

\begin{definition}
Let $\mathcal{E}$ be an $\infty$-category. For an étale $\infty$-stack $X$, we define the $\infty$-category of $\mathcal{E}$\textit{-valued sheaves} on $X$ as
\begin{align*}
\operatorname{Shv}(X;\mathcal{E}):=\operatorname{Shv}(\mathcal{LH}(X); \mathcal{E}),
\end{align*}
that is, the $\infty$-category of $\mathcal{E}$-valued sheaves on the étale site of $X$.  
\end{definition}

We will only be interested in sheaves valued in compactly generated $\infty$-categories.

\subsection{Descent results}

Recall from the previous section that a morphism $Y\rightarrow X$ of étale $\infty$-stacks is an \textit{open} (resp., \textit{closed}) embedding if for some (hence any) étale atlas $U\rightarrow X$, the pullback  $U \times_X Y\rightarrow U$ is an open (resp., closed) embedding. Likewise, a morphism $Y\rightarrow X$ of étale $\infty$-stacks is \textit{proper} if for some (hence any) étale atlas $U\rightarrow X$, the pullback $U\times_X Y\rightarrow U$ is a proper map of topological spaces.

\begin{definition}
Let $X$ be an étale $\infty$-stack. An \textit{open} (resp., \textit{closed}) \textit{substack} of $X$ is an isomorphism class of open (resp., closed) embeddings. 
\end{definition}

\begin{remark}
Let $X$ be an $\infty$-stack with étale atlas $\pi\colon U\rightarrow X$. An open (resp., closed) substack $V\hookrightarrow X$  gives rise to an open (resp., closed) subspace $W=\pi^{-1}(V)\subset U$ by pulling back along $\pi$, and this subspace satisfies that the pullbacks along the two projection maps $\operatorname{pr}_1,\operatorname{pr}_2\colon U\times_X U \rightarrow U$ coincide: $\operatorname{pr}_1^{-1}(W)=\operatorname{pr}_2^{-1}(W)\subset U\times_X U$. Conversely, an open (resp., closed) subspace $W\subset U$ satisfying that $\operatorname{pr}_1^{-1}(W)=\operatorname{pr}_2^{-1}(W)\subset U\times_X U$ defines an open (resp., closed) substack $V\hookrightarrow X$ as it defines an open (resp., closed) subsimplicial space of the \v{C}ech nerve $\check{C}(\pi)$.
\end{remark}

We have the following proper base-change theorem generalising \cite[Corollary 7.3.1.18]{LurieHTT} (see also \cite[Theorem 2.12]{ClausenOrsnesJansen}).

\begin{theorem}\label{proper base change}
Let
\begin{center}
\begin{tikzpicture}
\matrix (m) [matrix of math nodes,row sep=2em,column sep=2em]
  {
X' & X \\
Y' & Y \\
  };
  \path[-stealth]
(m-1-1) edge node[left]{$f'$} (m-2-1) edge node[above]{$g'$} (m-1-2)
(m-1-2) edge node[right]{$f$} (m-2-2) 
(m-2-1) edge node[below]{$g$} (m-2-2)
;
\end{tikzpicture}
\end{center}
be a pullback diagram of étale $\infty$-stacks with $f$ proper. Then the induced commutative diagram of $\infty$-categories resulting from applying $\operatorname{Shv}(-)$,
\begin{center}
\begin{tikzpicture}
\matrix (m) [matrix of math nodes,row sep=2em,column sep=2em]
  {
\operatorname{Shv}(X') & \operatorname{Shv}(X) \\
\operatorname{Shv}(Y') & \operatorname{Shv}(Y) \\
  };
  \path[-stealth]
(m-2-1) edge node[left]{$f'^*$} (m-1-1)
(m-1-2) edge node[above]{$g'^*$} (m-1-1)
(m-2-2) edge node[right]{$f^*$} (m-1-2) 
(m-2-2) edge node[below]{$g^*$} (m-2-1)
;
\end{tikzpicture}
\end{center}
is \emph{right adjointable} (or \emph{right Beck--Chevalley}): the vertical maps $f^*$ and $f'^*$ have right adjoints $f_*$ and $f'_*$, respectively, and the natural comparison map is an equivalence
\begin{align*}
g^*f_*\xrightarrow{\sim} f'_*g'^*
\end{align*}
\end{theorem}
\begin{proof}
Let $V\rightarrow Y$ be an étale atlas for $Y$, and let
\begin{align*}
U:=X\times_Y V\rightarrow X,\qquad U':=X'\times_Y V\rightarrow X',\qquad V':=Y'\times_Y V\rightarrow Y'
\end{align*}
denote the atlases obtained by pulling back $V\rightarrow Y$. The \v{C}ech nerves of these atlases give rise to a commutative square of $\Delta^{op}$-shaped diagrams of locally compact Hausdorff topological spaces:
\begin{center}
\begin{tikzpicture}
\matrix (m) [matrix of math nodes,row sep=2em,column sep=2em]
  {
U'_\bullet & U_\bullet \\
V'_\bullet & V_\bullet \\
  };
  \path[-stealth]
(m-1-1) edge node[left]{$f'_\bullet$} (m-2-1) edge node[above]{$g'_\bullet$} (m-1-2)
(m-1-2) edge node[right]{$f_\bullet$} (m-2-2) 
(m-2-1) edge node[below]{$g_\bullet$} (m-2-2)
;
\end{tikzpicture}
\end{center}

Now, the result follows directly from the proper base-change theorem for locally compact Hausdorff topological spaces (\cite[Theorem 2.12]{ClausenOrsnesJansen}) and \cite[Theorem 2.1.9]{AradCarmeliSchlank} applied to the diagram above.
\end{proof}

\begin{remark}\label{remark proper base-change}
Note that in the proof above, we explicitly use that our topological spaces are locally compact Hausdorff. Thus one should be careful to include this hypothesis if working with a category of topological spaces different from $\mathfrak{T}$ by requiring the existence of a locally compact Hausdorff atlas. Note also, however, that the proper base change theorem for topological spaces holds without the locally compact Hausdorff hypothesis if the proper map is also a closed inclusion (\cite[\S 7.3.2]{LurieHTT}), so in fact the results \Cref{sheaves on closed substack}, \Cref{cdh descent} and \Cref{descent for closed covers} below would hold true for some less restrictive choices of base category $\mathfrak{T}$. One should then be careful to make an appropriate definition of proper maps if the property of being proper is not local on the target. It may in general be possible to relax the definition of proper maps of topological $\infty$-stacks slightly and still obtain a proper base change result as above (see e.g. \cite[\S 4]{PortaYu16} and \cite[Theorem 6.8]{PortaYu21} for ideas in this direction). Note also that in the stratified setting, proper base change may hold for constructible sheaves even if it is not known to hold at the level of sheaves (\cite[Proposition 6.47]{PortaTeyssier}).
\end{remark}

Let us first of all use this result to prove the following lemma identifying sheaves on a closed substack. Note that for an open substack $j\colon U\hookrightarrow X$, \Cref{sheaves on open substack generalised} provides an equivalence
\begin{align*}
j_!\colon \operatorname{Shv}(U)\xrightarrow{\simeq}\operatorname{Shv}(X)_{/y(U)}
\end{align*}
that is compatible with base-change. The identification below is complementary to this (see also \cite[\S 7.2.3]{LurieHTT}).

\begin{lemma}\label{sheaves on closed substack}
Let $X$ be an étale $\infty$-stack and $i\colon Z\hookrightarrow X$ a closed substack with complement $U\hookrightarrow X$. The pushforward map yields an equivalence
\begin{align*}
i_*\colon \operatorname{Shv}(Z)\xrightarrow{\ \sim\ } \operatorname{ker}\big(\operatorname{Shv}(X)\rightarrow \operatorname{Shv}(U)\big)
\end{align*}
where the right hand side is the full subcategory spanned by the sheaves on $X$ which restrict to the terminal sheaf on $U$. Moreover, this identification is compatible with base-change: if $Z'\hookrightarrow X'$ is the pullback of $Z\hookrightarrow X$ along a morphism $f\colon X'\rightarrow X$ of $\infty$-stacks, then the induced diagram of sheaf categories below commutes, where $U'\hookrightarrow X'$ is the complement of $Z'$.
\begin{center}
\begin{tikzpicture}
\matrix (m) [matrix of math nodes,row sep=2em,column sep=2em]
  {
\operatorname{Shv}(Z) & \operatorname{ker}\big(\operatorname{Shv}(X)\rightarrow \operatorname{Shv}(U)\big) \\
\operatorname{Shv}(Z') & \operatorname{ker}\big(\operatorname{Shv}(X')\rightarrow\operatorname{Shv}(U')\big) \\
  };
  \path[-stealth]
(m-1-1) edge node[left]{$f'^*$} (m-2-1)
(m-1-2) edge node[above]{$i_*$} (m-1-1)
(m-1-2.220) edge node[right]{$f^*$} (m-2-2.140) 
(m-2-2) edge node[below]{$i'_*$} (m-2-1)
;
\end{tikzpicture}
\end{center}
\end{lemma}
\begin{proof}
The statement is true for a closed inclusion of topological spaces by \cite[Corollary 7.3.2.11]{LurieHTT}. The general statement for $\infty$-stacks then follows by taking the limit over the \v{C}ech nerves and noting that a sheaf on $X$ restricts to the terminal sheaf on $U$ if and only if it restricts to the terminal sheaf on the corresponding open subspace of an atlas. Compatibility with base-change follows from the proper base-change result above and the fact that the restriction functor, being left exact, preserves terminal objects.
\end{proof}

As in \cite{ClausenOrsnesJansen}, we use proper base-change to establish ``proper descent'' results.

\begin{corollary}\label{cdh descent}
Suppose
\begin{center}
\begin{tikzpicture}
\matrix (m) [matrix of math nodes,row sep=2em,column sep=2em]
  {
X' & X \\
Z & Y \\
  };
  \path[-stealth]
(m-1-1) edge node[left]{$f'$} (m-2-1) edge node[above]{$i'$} (m-1-2)
(m-1-2) edge node[right]{$f$} (m-2-2) 
(m-2-1) edge node[below]{$i$} (m-2-2)
;
\end{tikzpicture}
\end{center}
is a pullback square of étale $\infty$-stacks such that
\begin{enumerate}
\item $f$ is proper;
\item $i$ is a closed embedding;
\item the pullback of $f$ to the open complement $Y\smallsetminus Z$ is an isomorphism.
\end{enumerate}
Then applying $\operatorname{Shv}(-)$ with pullback functoriality yields a pullback diagram, so
\begin{align*}
\operatorname{Shv}(Y)\xrightarrow{\sim} \operatorname{Shv}(X)\mathop{\times}_{\operatorname{Shv}(X')}\operatorname{Shv}(Z)
\end{align*}
\end{corollary}
\begin{proof}
Let $j\colon U:=Y\smallsetminus Z\hookrightarrow X$ denote the open embedding of the complement of $Z$. By \Cref{sheaves on open substack generalised} and \Cref{sheaves on closed substack}, we have equivalences
\begin{align*}
i_*\colon \operatorname{Shv}(Z)\xrightarrow{\sim} \operatorname{ker}\big(\operatorname{Shv}(X)\rightarrow \operatorname{Shv}(U)\big)\quad\text{and}\quad j_!\colon \operatorname{Shv}(U)\xrightarrow{\sim} \operatorname{Shv}(X)_{/y(U)}.
\end{align*}
Hence by \cite[Lemma A.5.11]{LurieHA}, equivalence of sheaves on $X$ can be detected by pulling back to $Z$ and $U$. Moreover, these pullback functors preserve finite limits as they correspond to geometric morphisms of $\infty$-topoi. Therefore, it follows from \cite[Corollary 5.2.2.37]{LurieHTT} and proper base-change that we can test the conclusion of the statement after pulling the diagram back to $Z$ and $U$. But here it is obvious since on pullback to $Z$, the horizontal maps become equivalences and on pullback to $U$, the vertical maps become equivalences.
\end{proof}

By induction, we can generalise this to descent for closed covers.

\begin{corollary}\label{descent for closed covers}
Let $X$ be an étale $\infty$-stack and let $P$ be a finite set of closed substacks of $X$ viewed as a poset under inclusion. Suppose that for every subset $P'\subseteq P$, there is a subset $P''\subseteq P$ such that $\bigcap_{S\in P'}S = \bigcup_{T\in P''}T$; in other words, any intersection of elements in $P$ admits a cover by elements in $P$ (in particular, taking $P'=\emptyset$, we have $X=\bigcup_{S\in P}S$). Then
\begin{align*}
\operatorname{Shv}(X)\xrightarrow{\ \sim\ } \varprojlim_{S\in P^{op}} \operatorname{Shv}(S)
\end{align*}
via pullback.
\end{corollary}
\begin{proof}
If $P$ has less than $3$ elements, this is a special case of \Cref{cdh descent} where the proper map $f$ is also a closed inclusion. The general case follows by induction: for general $P$, assume that the claim holds for posets with strictly fewer elements, let $S_0\in P$ be a maximal element and set $X':= \bigcup_{S\in P\smallsetminus\{S_0\}} S$.

Recall that given an $\infty$-category $\mathcal{D}$, a subcategory $\mathcal{C}\subseteq \mathcal{D}$ is \textit{left closed} if the following condition holds: $d\rightarrow c$ and $c\in\mathcal{C}$ implies $d\in \mathcal{C}$. The categorical analogue of descent for closed covers (\Cref{descent for closed covers}) is descent for left closed full subcategories (\cite[Corollary 2.32]{ClausenOrsnesJansen}). By descent for left closed covers, $P$ identifies with the pushout in $\operatorname{Cat}_\infty$ of the span $P\smallsetminus \{S_0\} \leftarrow P_{\leq S_0}\smallsetminus\{S_0\} \rightarrow P_{\leq S_0}$. The claim then follows by the induction hypothesis and decomposition of limits over $P$ (\cite[Proposition 2.1]{ClausenOrsnesJansen}) as can be seen by analysing the following commutative diagram:
\begin{center}
\begin{tikzpicture}
\matrix (m) [matrix of math nodes,row sep=2em,column sep=2em]
  {
\operatorname{Shv}(X) & \operatorname{Shv}(X')\displaystyle\mathop{\times}_{\operatorname{Shv}(X'\cap S_0)} \operatorname{Shv}(S_0) \\
\displaystyle \varprojlim_{S\in P}\operatorname{Shv}(S) &\displaystyle \varprojlim_{\ \ S\in P\smallsetminus\{S_0\}}\!\!\!\!\!\!\operatorname{Shv}(S) \mathop{\times}_{\genfrac{}{}{0pt}{2}{\scriptstyle\varprojlim\ \ \operatorname{Shv}(S)}{\!\!\!\!\!\!\!\!\!\scriptscriptstyle S\in P_{\leq S_0}\smallsetminus \{S_0\}}} \displaystyle\varprojlim_{\ \ \,S\in P_{\leq S_0}}\!\!\! \operatorname{Shv}(S)\\
  };
  \path[-stealth]
(m-1-1) edge (m-2-1) edge (m-1-2.176)
(m-1-2) edge (m-2-2) 
(m-2-1.10) edge (m-2-2.174)
;
\end{tikzpicture}
\end{center}
This completes the proof.
\end{proof}

\section{Stratified homotopy theory}\label{stratified homotopy theory}

We now arrive at the main content of this note, namely setting up a stratified homotopy theory for stacks. As already mentioned, this is motivated by a concrete interest in the moduli stack of stable nodal curves, also known as the Deligne--Mumford--Knudsen compactification. Although these are just $1$-stacks, we strive to set up as simple a theory as possible by working in the realm of $\infty$-stacks and taking advantage of the tools made available to us by higher topos theory.

\subsection{Stratified \texorpdfstring{$\infty$}{infinity}-stacks and constructible sheaves}

Let $P$ be a poset. We consider $P$ as a topological space via its Alexandroff topology (the open sets are the upwards closed sets), and we consider it as an $\infty$-stack $P\in \operatorname{Shv}(\mathfrak{T})$ via the functor $S\mapsto \operatorname{Map}_{\operatorname{Top}}(S,P)$. Note that this is not generally representable by an object of $\mathfrak{T}$, as $P$ is Hausdorff if and only if it is discrete.

\begin{definition}
Let $X$ be an $\infty$-stack and let $P$ be a poset.
\begin{enumerate}
\item By a \textit{stratification} of $X$ (over $P$) we mean a morphism $s\colon X\rightarrow P$ in $\operatorname{Shv}(\mathfrak{T})$.
\item A \textit{stratified $\infty$-stack} is an $\infty$-stack $X$ equipped with a fixed stratification $s\colon X\rightarrow P$. We write $(X,s)$ or just $(X,P)$.
\item The \textit{strata} of a stratified $\infty$-stack $(X,P)$ are the substacks $X_p=X\times_P\{p\}$, $p\in P$.
\item A \textit{morphism} of stratified $\infty$-stacks $(X,s\colon X\rightarrow P)\rightarrow(Y,t\colon Y\rightarrow Q)$ consists of a morphism $f\colon X\rightarrow Y$ and a map of posets $\sigma\colon P\rightarrow Q$ such that $t\circ f=\sigma\circ s$.\qedhere
\end{enumerate}
\end{definition}

\begin{remark}
Unravelling the data of a morphism $X\rightarrow P$, we see that it is a collection $\{X_p\}_{p\in P}$ of pairwise disjoint locally closed substacks of $X$ such that for all $p\in P$, the union $X_{\geq p}=\bigcup_{q\geq p}X_q\subset X$ is an open substack and $\bigcup_{p\in P}X_p=X$.
\end{remark}

\begin{remark}
Suppose $(X,P)$ is a stratified $\infty$-stack and that $f\colon U\rightarrow X$ is an étale atlas. Then we can stratify $U$ over $P$ via the composite $U\rightarrow X\rightarrow P$. Likewise, we can stratify each level of the \v{C}ech nerve of $f$ over $P$; we denote the \textit{stratified \v{C}ech nerve} by $(\check{C}(f),P)$.
\end{remark}

We now define what it means for a sheaf on a (stratified) $\infty$-stack to be constant, locally constant and constructible (see also \cite[\S A.1 and \S A.5]{LurieHTT}).

\begin{definition} Let $\mathcal{C}$ be a compactly generated $\infty$-category, $X$ an étale $\infty$-stack and $\mathcal{F}$ a $\mathcal{C}$-valued sheaf on $X$.
\begin{enumerate}
\item $\mathcal{F}$ is \textit{constant} if it is in the essential image of the pullback functor $\mathcal{C}\simeq\operatorname{Shv}(*;\mathcal{C})\rightarrow \operatorname{Shv}(X; \mathcal{C})$ induced by the projection $X\rightarrow *$.

\item $\mathcal{F}$ is \textit{locally constant} if there exists a collection of morphisms $\iota_i\colon U_i\rightarrow X$, $i\in I$, in $\mathcal{LH}(X)$ such that the morphism $\iota\colon \coprod_{i\in I} U_i\rightarrow X$ is an effective epimorphism and the restriction $\iota_i^*(\mathcal{F})$ is constant in $\operatorname{Shv}(U_i;\mathcal{C})$ for each $i$.

\item If $X\rightarrow P$ is a stratification, then $\mathcal{F}$ is $P$-\textit{constructible} (or just \textit{constructible}) if its restriction to each stratum $X_p$ is locally constant.\qedhere
\end{enumerate}
\end{definition}

\begin{definition}
Let $(X,P)$ be a stratified étale $\infty$-stack and $\mathcal{C}$ a compactly generated $\infty$-category. We write $\operatorname{Shv}^{\operatorname{cbl}}_P(X;\mathcal{C})\subseteq \operatorname{Shv}(X;\mathcal{C})$ for the full subcategory spanned by the $P$-constructible sheaves; we omit the subscript $P$, if the poset is implicit. If $P=\ast$, then we write $\operatorname{Shv}^{\operatorname{LC}}(X;\mathcal{C})=\operatorname{Shv}^{\operatorname{cbl}}_{\ast}(X;\mathcal{C})$ for the full subcategory of \textit{locally constant} sheaves. If $\mathcal{C}=\mathcal{S}$, we will often omit the coefficients. 
\end{definition}

\begin{remark}
One readily verifies that our definition of locally constant sheaves recovers the one given by Lurie in \cite[Definition A.1.12]{LurieHA}. In the case of a topological space $X$ viewed as an $\infty$-stack, we can also see more explicitly that we recover the usual definition from the fact that any local homeomorphism can be refined to an open cover. In particular, if $(X,P)$ is a stratified topological space viewed as a stratified $\infty$-stack, then the above definition of constructible sheaves also recovers the usual one.
\end{remark}

We have the following simple but useful observation.

\begin{lemma}\label{locally constant and constructible sheaves via the Cech nerve}
Let $X$ be an $\infty$-stack with an étale atlas $f\colon U\rightarrow X$ and let $\check{C}(f)$ denote the \v{C}ech nerve of $f$. Then
\begin{align*}
\operatorname{Shv}^{\operatorname{LC}}(X)\xrightarrow{\ \sim\ } \varprojlim\operatorname{Shv}^{\operatorname{LC}}(\check{C}(f)).
\end{align*}
More generally, if $(X,P)$ is a stratified étale $\infty$-stack, then 
\begin{align*}
\operatorname{Shv}^{\operatorname{cbl}}_P(X)\xrightarrow{\ \sim\ } \varprojlim \operatorname{Shv}^{\operatorname{cbl}}_P(\check{C}(f)),
\end{align*}
where on the right hand side we consider the stratified \v{C}ech nerve $(\check{C}(f),P)$.
\end{lemma}
\begin{proof}
To prove the claim for locally constant sheaves, we need to show that a sheaf $\mathcal{F}$ on $X$ is locally constant if and only if it pulls back to a locally constant sheaf on each level of the \v{C}ech nerve. Thus all we need to verify is that if $f^*\mathcal{F}$ is locally constant on $U$, then $\mathcal{F}$ is itself locally constant. But this is immediate, since $f$ admits local sections. The general case of constructible sheaves follows by applying the above to the induced atlas of $p$-strata, $f_p\colon U_p\rightarrow X_p$, for each $p\in P$.
\end{proof}

Let's single out the following immediate application of descent for left closed covers (\Cref{descent for closed covers}) in the constructible setting.

\begin{corollary}\label{descent for closed covers, constructible setting}
Let $(X,P)$ be a stratified étale $\infty$-stack and consider the closed cover $\overline{X}_p$, $p\in P$, by closures of strata. Then
\begin{align*}
\operatorname{Shv}^{\operatorname{cbl}}_P(X)\xrightarrow{\ \sim\ } \varprojlim_{p\in P^{op}} \operatorname{Shv}^{\operatorname{cbl}}_{P_{\leq p}}(\overline{X}_p)
\end{align*}
via pullback.
\end{corollary}
\begin{proof}
By \Cref{descent for closed covers}, we have an equivalence
\begin{align*}
\operatorname{Shv}(X)\xrightarrow{\ \sim\ } \varprojlim_{p\in P^{op}} \operatorname{Shv}(\overline{X}_p).
\end{align*}
To see that it restricts to an equivalence of categories of constructible sheaves, it suffices to see that if a sheaf $\mathcal{F}$ on $X$ pulls back to a constructible sheaf on each $\overline{X}_p$, then $\mathcal{F}$ is itself constructible. But this is obvious as each stratum $X_p$ is contained in its own closure.
\end{proof}

Likewise, one can apply ``cdh descent'' (\Cref{cdh descent}) to the stratified setting in the following situation; we leave the proof to the reader.

\begin{corollary}
Let $f\colon (X,P)\rightarrow (Y,P)$ be a morphism of stratified étale $\infty$-stacks whose underlying map is proper. Assume moreover that $P$ has a maximal element $m$ and that $f$ restricts to an isomorphism over the open stratum $Y_m$. Then 
\begin{align*}
\operatorname{Shv}_P^{\operatorname{cbl}}(Y)\xrightarrow{\sim} \operatorname{Shv}_P^{\operatorname{cbl}}(X) \mathop{\times}_{\operatorname{Shv}_{P\smallsetminus m}^{\operatorname{cbl}}(X\smallsetminus X_m)}\operatorname{Shv}_{P\smallsetminus m}^{\operatorname{cbl}}(Y\smallsetminus Y_m)
\end{align*}
via pullback.
\end{corollary}

The following identification implies that generally it suffices to consider the category of $\mathcal{S}$-valued sheaves.

\begin{lemma}
Let $\mathcal{C}$ be a compactly generated $\infty$-category and $\mathcal{C}_0$ the subcategory of compact objects. For a stratified étale $\infty$-stack $(X,P)$, we have an equivalence
\begin{align*}
\operatorname{Shv}^{\operatorname{cbl}}_P(X;\mathcal{C})\xrightarrow{\ \sim \ } \operatorname{Fun}^{\operatorname{lex}}(\mathcal{C}_0^{\operatorname{op}},\operatorname{Shv}^{\operatorname{cbl}}_P(X;\mathcal{S}))
\end{align*}
\end{lemma}
\begin{proof}
See the proof of \cite[Lemma B.3]{OrsnesJansen}.
\end{proof}

\subsection{Exit path \texorpdfstring{$\infty$}{infinity}-categories}

We now introduce exit path $\infty$-categories of $\infty$-stacks. As mentioned in the introduction, we opt for a non-constructive approach: essentially, the exit path $\infty$-category, \textit{if it exists}, is the idempotent complete $\infty$-category classifying constructible sheaves. In this way, the classification of constructible sheaves comes for free at the cost of not having a concrete model for the exit path $\infty$-category. 

We recall the notion of atomic objects which provide a useful characterisation of presentable $\infty$-categories of the form $\mathcal{P}(\mathcal{C})$ for some $\infty$-category $\mathcal{C}$. We refer to \cite[\S 2.2]{ClausenOrsnesJansen} for details. An object $x$ of a presentable $\infty$-category $\mathcal{D}$ is \textit{atomic} if the functor $\operatorname{Map}(x,-)\colon \mathcal{D}\rightarrow \mathcal{S}$ commutes with \textit{all} colimits. We write $\mathcal{D}^{atom}\subset \mathcal{D}$ for the full subcategory of atomic objects. Consider the functor 
\begin{align*}
\operatorname{Cat}_\infty\rightarrow \operatorname{Pr}^L
\end{align*}
sending an $\infty$-category $\mathcal{C}$ to the presheaf $\infty$-category $\mathcal{P}(\mathcal{C})$ and a morphism $f\colon \mathcal{C}\rightarrow \mathcal{D}$ to the unique colimit preserving functor $\mathcal{P}(C)\rightarrow \mathcal{P}(D)$ restricting via the Yoneda embedding to the composite $y_D\circ f\colon \mathcal{C}\rightarrow \mathcal{D}\rightarrow \mathcal{P}(\mathcal{D})$; equivalently, $f$ is sent to the left adjoint of restriction along $f$ (\cite[Corollary F]{HaugsengHebestreitLinskensNuiten}). The characterisation of presentable $\infty$-categories of the form $\mathcal{P}(\mathcal{C})$ is supplied by the following proposition (\cite[Proposition 2.7]{ClausenOrsnesJansen}).

\begin{proposition}\label{presheaf categories and atomic objects}
The functor $\operatorname{Cat}_\infty\rightarrow \operatorname{Pr}^L$, $\mathcal{C}\mapsto \mathcal{P}(\mathcal{C})$, restricts to an equivalence between the $\infty$-category of idempotent-complete small $\infty$-categories and the subcategory of $\operatorname{Pr}^L$ whose objects are the presentable $\infty$-categories generated by atomic objects and whose morphisms are the colimit-preserving functors whose right adjoint also preserves colimits. The inverse functor is given by $\mathcal{D}\mapsto \mathcal{D}^{atom}$.
\end{proposition}

If $P$ is a poset satisfying the ascending chain condition and $(X,\pi\colon X\rightarrow P)$ is a stratified étale $\infty$-stack, then we have a geometric morphism of $\infty$-topoi
\begin{align*}
\pi_*\colon \operatorname{Shv}(X)\rightarrow \operatorname{Fun}(P,\mathcal{S}),\quad (\pi_*\mathcal{F})(p)=\mathcal{F}(X_{\geq p}),
\end{align*}
and an induced comparison functor
\begin{align*}
\pi^*\colon \operatorname{Fun}(P,\mathcal{S})\rightarrow \operatorname{Shv}_P^{\operatorname{cbl}}(X).
\end{align*}
(see the discussion before \cite[Theorem 3.4]{ClausenOrsnesJansen}). In line with \cite{ClausenOrsnesJansen}, we make the following definition.

\begin{definition}\label{definition exit path category}
Let $P$ be a poset satisfying the ascending chain condition.
\begin{enumerate}
\item We say that a stratified étale $\infty$-stack $(X, P)$ \textit{admits an exit path $\infty$-category} if the following conditions hold:
\begin{enumerate}
\item The full subcategory $\operatorname{Shv}^{\operatorname{cbl}}_P(X)\subset\operatorname{Shv}(X)$ is closed under all limits and colimits.
\item The $\infty$-category $\operatorname{Shv}^{\operatorname{cbl}}_P(X)$ is generated under colimits by a set of atomic objects.
\item The pullback $\pi^*\colon \operatorname{Fun}(P, \mathcal{S})\rightarrow \operatorname{Shv}^{\operatorname{cbl}}_P(X)$ preserves all limits (and colimits, but that is automatic).
\end{enumerate}
\item If $(X, P)$ admits an exit path $\infty$-category, we define its \textit{exit path $\infty$-category} to be the opposite category of the full subcategory of atomic constructible sheaves:
\begin{align*}
\Pi(X,P):=\left(\operatorname{Shv}^{\operatorname{cbl}}_P(X)^{atom}\right)^{\operatorname{op}}.
\end{align*}
\item If $f\colon (X,P)\rightarrow (Y,Q)$ is a map of stratified étale $\infty$-stacks that admit exit path $\infty$-categories, we say that $f$ \textit{respects exit path $\infty$-categories} if the pullback functor
\begin{align*}
f^*\colon \operatorname{Shv}^{\operatorname{cbl}}_Q(Y)\rightarrow \operatorname{Shv}^{\operatorname{cbl}}_P(X)
\end{align*}
preserves limits (and colimits, but that is automatic).\qedhere
\end{enumerate}
\end{definition}

If $(X,P)$ admits an exit path $\infty$-category, then it follows from \Cref{presheaf categories and atomic objects} that there is an induced ``exodromy'' equivalence (cf.~\cite{BarwickGlasmanHaine} for the terminology)
\begin{align*}
\operatorname{Fun}(\Pi(X,P),\mathcal{S})\xrightarrow{\ \sim \ } \operatorname{Shv}^{\operatorname{cbl}}_P(X).
\end{align*}

It also follows from \Cref{presheaf categories and atomic objects} that condition (c) of part (1) above implies that the pullback map
\begin{align*}
\pi^*\colon \operatorname{Fun}(P,\mathcal{S})\rightarrow \operatorname{Shv}^{\operatorname{cbl}}_P(X)
\end{align*}
is recovered by precomposition with a uniquely determined functor $\Pi(X,P)\rightarrow P$.

More generally, by \Cref{presheaf categories and atomic objects}, the condition that $f\colon (X,P)\rightarrow (Y,Q)$ respects exit path $\infty$-categories is equivalent to the condition that the induced pullback functor on constructible sheaves, $f^*\colon \operatorname{Shv}^{\operatorname{cbl}}_Q(Y)\rightarrow \operatorname{Shv}^{\operatorname{cbl}}_P(X)$, is given, via the exodromy equivalence above, by composition with a functor
\begin{align*}
\Pi(X,P)\rightarrow \Pi(Y,Q).
\end{align*}
What is more, this functor is uniquely determined as the restriction to atomic objects of the left adjoint to $f^*\colon \operatorname{Shv}^{\operatorname{cbl}}_Q(Y)\rightarrow \operatorname{Shv}^{\operatorname{cbl}}_P(X)$.

\begin{example}\label{Lurie exodromy}
If a stratified topological space $(X,P)$ is paracompact, locally of singular shape (e.g.~locally contractible) and conically stratified over a poset satisfying the ascending chain condition, then it admits an exit path $\infty$-category in the sense of (1) above: conditions (b) and (c) follow from the work of Lurie (\cite[Theorem A.9.3]{LurieHA}) and condition (a) is verified in \cite[Corollary 5.20]{PortaTeyssier}. Moreover, the exit path $\infty$-category can be obtained as a concrete quasi-category (\cite[Definition A.6.2]{LurieHA}, see \cite[\S 2]{OrsnesJansen} for a short overview).

Lurie also shows that any stratum preserving map $f\colon (Y,P)\rightarrow (X,P)$ between two such stratified spaces (stratified over the same poset $P$) will preserve exit path $\infty$-categories by \cite[Proposition A.9.16]{LurieHTT} (see also \Cref{all maps are exodromic} below).

A concrete instance of this example is the reductive Borel--Serre compactification of a locally symmetric space associated to a neat arithmetic group (see \cite{OrsnesJansen} and \cite{ClausenOrsnesJansen}). An example that goes beyond conically stratified topological \textit{spaces} is that of the Deligne--Mumford--Knudsen compactification, also known as the moduli stack of stable nodal curves (see \cite{OrsnesJansen23a}). We will comment more on this at the end of this section.
\end{example}

\begin{observation}\label{maps of stratifying posets}
Note that if $(X,P)$ and $(Y,Q)$ are stratified $\infty$-stacks whose exit path $\infty$-categories identify with the stratifying posets, then any morphism $(X,P)\rightarrow (Y,Q)$ respects exit path $\infty$-categories as the required map $P\rightarrow Q$ is part of the data.
\end{observation}

Together with the permanence properties of \Cref{permanence properties} below, \Cref{maps of stratifying posets} provides a powerful tool for identifying exit path $\infty$-categories.

\begin{remark}\label{all maps are exodromic}
After this note was originally written, Haine--Porta--Teyssier made the remarkable observation that in fact condition (3) in the definition above is automatic (\cite[Theorem 3.2.3]{HainePortaTeyssier24}). More precisely, any morphism of stratified $\infty$-stacks admitting exit path $\infty$-categories \textit{must} preserve these!
\end{remark}

The result below is a direct generalisation of \cite[Proposition 3.6]{ClausenOrsnesJansen} and the proof is completely analogous.

\begin{proposition}\label{permanence properties}\ 
\begin{enumerate}
\item Let $P$ be a poset satisfying the ascending chain condition and let $(X,P)$ be a stratified étale $\infty$-stack admitting an exit path $\infty$-category. For a locally closed subset $Q\subset P$ consider the substack $X_Q=X\times_P Q$, naturally stratified over $Q$. The stratified stack $(X_Q,Q)$ admits an exit path $\infty$-category, the inclusion $(X_Q, Q)\hookrightarrow (X,P)$ respects $\infty$-categories, and
\begin{align*}
\Pi(X_Q,Q)\xrightarrow{\ \sim\ } \Pi(X,P)\mathop{\times}_{P} Q.
\end{align*}
\item Let $K$ be a small $\infty$-category and $\{(X_k, P_k)\}_{k\in K}$ a $K$-shaped diagram of stratified étale $\infty$-stacks equipped with a co-cone $(X_\infty, P_\infty)$. Suppose:
\begin{enumerate}[label=\roman*)]
\item $P_k$ satisfies the ascending chain condition for each $k\in K$ and so does $P_\infty$;
\item $(X_k, P_k)$ admits an exit path $\infty$-category for each $k\in K$;
\item for every $k\rightarrow k'$, the map $(X_k, P_k)\rightarrow (X_{k'}, P_{k'})$ respects exit path $\infty$-categories;
\item we have $\operatorname{Shv}(X_\infty)\xrightarrow{\sim}\varprojlim_{k\in K^{op}}\operatorname{Shv}(X_k)$ and $\operatorname{Shv}^{\operatorname{cbl}}_{P_\infty}(X_\infty)\xrightarrow{\sim}\varprojlim_{k\in K^{op}}\operatorname{Shv}^{\operatorname{cbl}}_{P_k}(X_k)$.
\end{enumerate}
Then
\begin{itemize}
\item $(X_\infty, P_\infty)$ admits an exit path $\infty$-category;
\item the map $(X_k, P_k)\rightarrow (X_\infty, P_\infty)$ respects exit path $\infty$-categories for every $k\in K$;
\item $\Pi(X_\infty, P_\infty)\xleftarrow{\sim}\varinjlim_{k\in K}\Pi(X_k, P_k)$.
\end{itemize}
\end{enumerate}
\end{proposition}
\begin{proof}
Factoring a locally closed inclusion as a closed inclusion followed by an open inclusion, it suffices to treat those two cases separately in order to prove part 1. For $Q\subset P$ open, set $U=X_Q\subset X$. Then $\operatorname{Shv}(U)\xrightarrow{\sim} \operatorname{Shv}(X)_{/y(U)}$ by \Cref{sheaves on open substack generalised} and since this description is compatible with pullback, it passes to constructible sheaves: $\operatorname{Shv}_Q^{\operatorname{cbl}}(U)\xrightarrow{\sim} \operatorname{Shv}_P^{\operatorname{cbl}}(X)_{/y(U)}$. This verifies conditions (a) and (b) of \Cref{definition exit path category} (1), and doing the same thing with the open subset $Q\subset P$, we can also verify condition (c), so $(U, Q)$ does indeed admit an exit path $\infty$-category. The claim that the inclusion $(U, Q)\hookrightarrow (X, P)$ preserves exit path $\infty$-categories also follows from these identifications. Now, the explicit identification of the exit path $\infty$-category follows from \Cref{presheaf categories and atomic objects} and the equivalence
\begin{align*}
\operatorname{Fun}(\Pi(X,P)\mathop{\times}_{P}Q,\mathcal{S})\xrightarrow{\sim} \operatorname{Fun}(\Pi(X,P),\mathcal{S})_{/F_U}
\end{align*}
given by restriction and left Kan extension along $\Pi(X,P)\mathop{\times}_{P}Q\rightarrow \Pi(X,P)$, and where the functor $F_U\colon \Pi(X,P)\rightarrow \mathcal{S}$ is the one that restricts to the terminal functor over $\Pi(X,P)\mathop{\times}_{P}Q$ and to the initial functor on the complement.

For $Q\subset P$ a closed subset, we argue similarly, but now for $Z=X_Q\subset X$ with complement $U\subseteq X$, we use the identification $\operatorname{Shv}(Z)\xrightarrow{\sim} \operatorname{ker}(\operatorname{Shv}(X)\rightarrow \operatorname{Shv}(U))$ of \Cref{sheaves on closed substack}. The final identification of the exit path category now follows from \Cref{presheaf categories and atomic objects} and the equivalence
\begin{align*}
\operatorname{Fun}(\Pi(X,P)\mathop{\times}_{P}Q,\mathcal{S})\xrightarrow{\sim}\operatorname{ker}\big( \operatorname{Fun}(\Pi(X,P),\mathcal{S})\rightarrow \operatorname{Fun}(\Pi(U,Q),\mathcal{S})\big)
\end{align*}
given by restriction and right Kan extension along $\Pi(X,P)\mathop{\times}_{P}Q\rightarrow \Pi(X,P)$, and where the right hand side consists of the functors out of $\Pi(X,P)$ which restrict to the terminal functor on $\Pi(U,Q)$.

Part 2 follows more or less directly from \Cref{presheaf categories and atomic objects}. Combined with assumptions (ii), (iii) and (iv), we see that conditions (a) and (b) of \Cref{definition exit path category} (1) are satisfied. Condition (c) follows from the additional observation that each $\operatorname{Fun}(P_\infty,\mathcal{S})\rightarrow\operatorname{Fun}(P_k,\mathcal{S})$ preserves limits. Thus we conlude that $X_\infty\rightarrow P_\infty$ does indeed admit an exit path $\infty$-category. The claim that the maps $(X_k, P_k)\rightarrow (X_\infty, P_\infty)$ respect exit path $\infty$-categories follows from assumptions (ii), (iii) and (iv). The equivalence $\Pi(X_\infty, P_\infty)\xleftarrow{\sim}\varinjlim_{k\in K}\Pi(X_k, P_k)$ is a direct consequence of \Cref{presheaf categories and atomic objects}.
\end{proof}

Let us apply this immediately to show that the exit path $\infty$-category of a stratified stack can be calculated using a suitably nice atlas. This follows from the observation made in \Cref{locally constant and constructible sheaves via the Cech nerve}, and although it may not be a useful calculational tool in itself, it tells us that our definition agrees with another natural candidate for defining the exit path $\infty$-category of a stratified stack, namely by using the \v{C}ech nerve of an atlas.

\begin{corollary}\label{exit path category in terms of an atlas}
Let $(X,P)$ be a stratified $\infty$-stack with an étale atlas $f\colon U\rightarrow X$, and suppose $P$ satisfies the ascending chain condition. Let $\check{C}(f)\colon [n]\mapsto (U_n,P)$ denote the stratified \v{C}ech nerve of the atlas. If each $(U_n,P)$ admits an exit path $\infty$-category, and all the induced maps $(U_n, P) \rightarrow (U_m, P)$ preserve exit path $\infty$-categories, then $(X, P)$ admits an exit path $\infty$-category, the (stratified) atlas morphism $f\colon (U,P)\rightarrow (X,P)$ preserves exit path $\infty$-categories, and
\begin{align*}
\Pi(X,P)\xleftarrow{\ \sim \ } \varinjlim_{[n]\in \Delta^{op}} \Pi(U_n,P).
\end{align*}
\end{corollary}

\begin{remark}
It is somewhat unsatisfying that we have to ensure that all levels of the \v{C}ech nerve admit exit path $\infty$-categories and that all the structure maps preserve exit path $\infty$-categories --- this is because the property of admitting an exit path $\infty$-category is not well-behaved under arbitrary pullbacks. In practice, however, it is rarely so bad. If for example each $U_n$ is paracompact, locally of singular shape and conically stratified over $P$, then by \Cref{Lurie exodromy}, they admit exit path $\infty$-categories, and all structure maps preserve exit path $\infty$-categories.
\end{remark}

We can also use \Cref{permanence properties} to treat certain quotient stacks. Let $(X,P)$ be a stratified $\infty$-stack with $P$ satisfying the ascending chain condition and let $G$ be a group acting compatibly on $X$ and $P$, i.e. such that the stratification map is equivariant: $x\in X_p\Rightarrow g.x\in X_{g.p}$ for all $g\in G$. We say that $G$ acts on the stratified stack $(X,P)$. Since $P$ satisfies the ascending chain condition, the following implication must hold:
\begin{align*}
p\leq g.p\Rightarrow p=g.p\quad\text{ for all }p\in P,\ g\in G
\end{align*}
It follows that there is a natural poset structure on the quotient set $G\backslash P$ identifying it with the quotient in the category of posets: $[p]\leq [q]$ if there is a $g\in G$ such that $g.p\leq q$ in $P$. Moreover, the stratification of $X\rightarrow P$ naturally descends to define a stratification of the quotient $\infty$-stack $[G\backslash X]\rightarrow G\backslash P$ (\cite[Lemma 2.34]{ClausenOrsnesJansen}). Note also that if $G$ is a group acting on a poset $P$, then the colimit $\varinjlim_{BG} P$ in $\operatorname{Cat}_\infty$ naturally identifies with the action category $G\backslash\backslash P$ with objects $p\in P$ and whose morphisms $p\rightarrow p'$ are the $g\in G$ such that $g.p\leq p'$ (\cite[Corollary 2.34]{ClausenOrsnesJansen}).

\begin{corollary}
Let $(X,P)$ be a stratified étale $\infty$-stack with $P$ satisfying the ascending chain condition, let $G$ be a group acting on $(X,P)$. If $(X,P)$ admits an exit path $\infty$-category, then so does the resulting quotient stack $([G\backslash X], G\backslash P)$, the quotient map preserves exit path $\infty$-categories, and moreover we have an equivalence:
\begin{align*}
\Pi([G\backslash X],G\backslash P)\xleftarrow{\ \sim \ } \varinjlim_{BG} \Pi(X,P).
\end{align*}
In particular, if $\Pi(X,P)\xrightarrow{\sim} P$, then the exit path $\infty$-category of the quotient stack naturally identifies with the action category, $\Pi([G\backslash X],G\backslash P)\xleftarrow{\sim}G\backslash\backslash P$.
\end{corollary}
\begin{proof}
We readily see that conditions (i)-(iii) of \Cref{permanence properties} (2) hold. Applying $\operatorname{Shv}$ is colimit preserving, so we have an equivalence $\operatorname{Shv}([G\backslash X])\xrightarrow{\sim} \varprojlim_{BG} \operatorname{Shv}(X)$. Thus all we need to verify is that this equivalence restricts to an equivalence of the full subcategories of constructible sheaves, i.e.~we must show that if a sheaf $\mathcal{F}$ on $[G\backslash X]$ pulls back to a constructible sheaf on $X$, then $\mathcal{F}$ is itself constructible. But this is clear since for each $p\in P$, the projection $X_p\rightarrow [G\backslash X]_{[p]}$ admits local sections; indeed, the $[p]$-stratum identifies with the quotient of the $p$-stratum $X_p$ by the stabiliser $G_p$ of $p$, $[G\backslash X]_{[p]}\simeq [G_p\backslash X_p]$.
\end{proof}

\begin{remark}
If $X$ is a topological space with a group action by $G$, then we have a quotient $1$-stack classifying principal $G$-bundles $\pi\colon E\rightarrow B$ together with an equivariant map $\sigma\colon E\rightarrow X$ (see e.g.~\cite[p.286]{ArbarelloCornalbaGriffiths}). This agrees with the $\infty$-quotient stack $[G\backslash X]=\varinjlim_{BG}X$ as they both admit an atlas morphism from $X$ and the corresponding \v{C}ech nerves agree (see \cite{NikolausSchreiberStevenson}). Note, however, that if we work only with stratified topological spaces without moving into the realm of stacks, then one needs to be more careful when considering group actions, see for example \cite[Corollary 2.18 and Proposition 4.8]{ClausenOrsnesJansen}.
\end{remark}

As a more concrete application of \Cref{permanence properties}, we refer to the calculation in \cite{OrsnesJansen23a} where we identify the exit path $\infty$-category of the Deligne--Mumford--Knudsen compactification of the moduli stack of stable curves. The calculational strategy follows that of \cite{ClausenOrsnesJansen}, where we identify the exit path $\infty$-category of the reductive Borel--Serre compactification of a locally symmetric space associated to a \textit{neat} arithmetic group (the identification can also be found in \cite{OrsnesJansen}, but the calculational strategies are very different). The main idea is to exhibit a co-cone as in \Cref{permanence properties}, where the exit path $\infty$-category of each $(X_k, P_k)$, $k\in K$, identifies with the stratifying poset $P_k$ --- this should be interpreted as the analogue of being contractible for stratified spaces.

%

\subsection{Homotopy type, refinement and localisations}

If we consider the trivial stratification over the terminal poset $P = \ast$, then constructible sheaves are just locally constant sheaves and by comparing with \cite[\S A.1]{LurieHA}, we find that a (stratified) $\infty$-stack $(X,\ast)$ admits an exit path $\infty$-category if and only if $\operatorname{Shv}(X)$ is locally of constant shape in the sense of \cite[Definition A.1.5]{LurieHA} (see also Proposition A.1.8). In this case, the exit path $\infty$-category is the $\infty$-groupoid given by the shape (\cite[Theorem A.1.15, Remark A.1.10]{LurieHA} and the straightening-unstraightening equivalence). This defines the \textit{homotopy type} (or \textit{fundamental $\infty$-groupoid}), $\Pi(X)=\Pi(X,\ast)$, of an étale $\infty$-stack $X$ whose $\infty$-category of sheaves is locally of constant shape. By \Cref{exit path category in terms of an atlas}, it can be determined as a colimit using a suitable atlas:
\begin{align*}
\Pi(X)\xleftarrow{\ \sim  \ }\varinjlim_{[n]\in \Delta^{op}} \Pi(U_n),
\end{align*}
where $U_n=\check{C}(f)_n$ is the $n$'th level of the \v{C}ech nerve of an étale atlas $f\colon U\rightarrow X$ satisfying that each $\operatorname{Shv}(U_n)$ is locally of constant shape. For a suitable choice of atlas, this recovers the definition of the homotopy type of a ($1$-truncated) stack as the (fat) geometric realisation of the \v{C}ech nerve given in e.g.~\cite{Ebert,Noohi,Moerdijk}.

Combining the above observation that the exit path $\infty$-category of a trivially stratified $\infty$-stack $(X,\ast)$ is an $\infty$-groupoid with \Cref{permanence properties} (1), the following lemma is immediate.

\begin{lemma}
Let $(X,P)$ be a stratified étale $\infty$-stack admitting an exit path $\infty$-category. The functor $\Pi(X,P)\rightarrow P$ is conservative.
\end{lemma}

This observation allows us to prove the following result, which can be phrased as follows: the exit path $\infty$-category sends refinements to localisations (see also \cite[Theorem 3.3.12]{AyalaFrancisRozenblyum}).

\begin{proposition}
Let $\sigma\colon P\rightarrow Q$ be a map of posets satisfying the ascending chain condition. Let $(X,P)$ be a stratified étale $\infty$-stack and let $(X,Q)$ denote the same stack stratified over $Q$ via the map $\sigma$. If both $(X,P)$ and $(X,Q)$ admit exit path $\infty$-categories, then the map $(X,P)\rightarrow (X,Q)$ given by the identity on $X$ preserves exit path $\infty$-categories and the corresponding map $\Pi(X,P)\rightarrow \Pi(X,Q)$ is a localisation.
\end{proposition}
\begin{proof}
The fact that $(X,P)\rightarrow (X,Q)$ preserves exit path $\infty$-categories follows directly from \Cref{definition exit path category} (1) (a). Now, we claim that the proof of \cite[Theorem 3.3.12]{AyalaFrancisRozenblyum} goes through to prove that the map of exit path $\infty$-categories is a localisation. The proof uses the subsequent Lemma 3.3.14 loc.cit. exploiting the fact that $P$ satisfies the ascending chain condition and that the map $\Pi(X,P)\rightarrow P$ is conservative.
\end{proof}

In particular we have the following corollary, which should be compared with \cite[Corollary A.9.4]{LurieHA}.

\begin{corollary}
Let $(X,P)$ be stratified étale $\infty$-stack. Suppose that both $(X,P)$ and $(X,\ast)$ admit exit path $\infty$-categories. Then the functor
\begin{align*}
\Pi(X,P)\rightarrow \Pi(X)
\end{align*}
is a weak homotopy equivalence.
\end{corollary}

\subsection{The constructible derived category}

We remark here that exodromy for constructible sheaves with values in $\mathcal{S}$ automatically extends to sheaves with values in an arbitrary compactly generated $\infty$-category.

\begin{proposition}
Let $(X,P)$ be a stratified étale $\infty$-stack admitting an exit path $\infty$-category, and let $\mathcal{C}$ be a compactly generated $\infty$-category. Then there is an equivalence
\begin{align*}
\operatorname{Fun}(\Pi(X, P), \mathcal{C})\xrightarrow{\ \sim \ }\operatorname{Shv}^{\operatorname{cbl}}_P(X;\mathcal{C}).
\end{align*}
\end{proposition}
\begin{proof}
The proof given in \cite[Appendix B]{OrsnesJansen} applies to stratified topological spaces, but it goes through for stratified $\infty$-stacks by replacing the site $\mathscr{U}(-)$ of open sets by the site $\mathcal{LH}(-)$ of local homeomorphisms.
\end{proof}

This generalisation implies that we can access the constructible derived category of a stratified stack admitting an exit path $\infty$-category by considering sheaves valued in $\mathscr{D}(R)\simeq \operatorname{LMod}_R$, the derived $\infty$-category of left $R$-modules for a given associative ring $R$. See also the discussion in \cite[\S 2.4]{OrsnesJansen}. We refer to \cite[\S 1.3]{LurieHA} for details on derived $\infty$-categories.

Let $(X,P)$ be a stratified étale $\infty$-stack and let $\operatorname{Shv}_1(X;R):=\operatorname{Shv}_1(\mathcal{LH}(X); \operatorname{LMod}^1_R)$ denote the $1$-category of sheaves of left $R$-modules on $X$, i.e.~sheaves on the site $\mathcal{LH}(X)$ with values in the $1$-category of left $R$-modules. Consider the canonical (and fully faithful) functor
\begin{align*}
\mathscr{D}(\operatorname{Shv}_1(X;R))\longrightarrow \operatorname{Shv}(X;\operatorname{LMod}_R).
\end{align*}
The essential image of this functor is the subcategory of hypercomplete sheaves by which we mean the subcategory given by tensoring $\operatorname{LMod}_R$ with the subcategory $\operatorname{Shv}^{\operatorname{hyp}}(X)$ of hypercomplete objects in the $\infty$-topos $\operatorname{Shv}(X)$ (see e.g.~\cite[\S 6.5.2]{LurieHTT} and \cite{HainePortaTeyssier23}); equivalently, it corresponds to the subcategory $\operatorname{Fun}^{\operatorname{lex}}(\operatorname{Perf}(R), \operatorname{Shv}^{\operatorname{hyp}}(X))$ under the natural equivalence of \Cref{sheaves with general coefficients from space valued sheaves}. If $(X,P)$ admits an exit path $\infty$-category, then the constructible sheaves on $X$ are hypercomplete, as presheaf $\infty$-topoi have enough points (\cite[Remark 6.5.4.7]{LurieHTT}). Moreover, under the above functor, the subcategory of constructible sheaves exactly corresponds to the full subcategory
\begin{align*}
\mathscr{D}_P(\operatorname{Shv}_1(X;R))\subseteq\mathscr{D}(\operatorname{Shv}_1(X;R))
\end{align*}
spanned by the complexes whose homology sheaves are constructible with respect to the given stratification over $P$. The full subcategory of constructible compact-valued sheaves (that is, sheaves whose stalk complexes are perfect) is denoted by
\begin{align*}
\operatorname{Shv}^{\operatorname{cbl,cpt}}_P(X;\operatorname{LMod}_R)\subseteq \operatorname{Shv}^{\operatorname{cbl}}_P(X;\operatorname{LMod}_R)
\end{align*}
and corresponds, under the above functor, to the full subcategory
\begin{align*}
\mathscr{D}^b_P(\operatorname{Shv}_1(X;R))\subseteq\mathscr{D}_P(\operatorname{Shv}_1(X;R))
\end{align*}
spanned by the complexes whose homology sheaves are constructible \textit{and} whose stalk complexes are perfect, i.e.~the bounded constructible derived $\infty$-category (with respect to the fixed stratification over $P$).

We have the following result identifying the (bounded) constructible derived category as a derived functor category when the exit path $\infty$-category is just a $1$-category (see \cite[Theorem 2.28]{OrsnesJansen} for a proof).

\begin{proposition}\label{equivalence of derived infinity-categories}
Let $(X,P)$ be a stratified étale $\infty$-stack admitting an exit path $\infty$-category, and let $R$ be an associative ring. If $\Pi(X, P)$ is a $1$-category, then there is an equivalence
\begin{align*}
\operatorname{Shv}^{\operatorname{cbl}}_P(X;\operatorname{LMod}_R)\xrightarrow{\sim}\mathscr{D}(\operatorname{Fun}(\Pi(X, P),\operatorname{LMod}_R^1)),
\end{align*}
which restricts to an equivalence
\begin{align*}
\operatorname{Shv}^{\operatorname{cbl,cpt}}_P(X;\operatorname{LMod}_R)\xrightarrow{\sim} \mathscr{D}^{\operatorname{cpt}}(\operatorname{Fun}(\Pi(X,P),\operatorname{LMod}_R^1)),
\end{align*}
where $\mathscr{D}^{\operatorname{cpt}}(\operatorname{Fun}(\Pi(X, P),\operatorname{LMod}_R^1))$ is the full subcategory spanned by the complexes of functors $F_\bullet$ such that $F_\bullet(x)$ is a perfect complex for all objects $x$ of $\Pi(X,P)$.
\end{proposition}

Taking homotopy categories, we get the following corollary, where $D_P(-)$ denotes the full subcategory of the $1$-categorical derived category of sheaves of $R$-modules spanned by the complexes with $P$-constructible cohomology sheaves and $D^b_P(-)$ the further full subcategory spanned by those complexes all of whose stalk complexes are perfect. In other words, $D^b_P(-)$ is the ``classical'' constructible derived ($1$-)category (beware, however, that we are still working with respect to a fixed stratification).

\begin{corollary}
In the situation of \Cref{equivalence of derived infinity-categories} there is an equivalence of $1$-categories
\begin{align*}
D_P(\operatorname{Shv}_1(X;R))\simeq D(\operatorname{Fun}(\Pi(X,P),\operatorname{LMod}_R^1))
\end{align*}
which restricts to an equivalence
\begin{align*}
D^b_P(\operatorname{Shv}_1(X;R))\simeq D^{\operatorname{cpt}}(\operatorname{Fun}(\Pi(X,P),\operatorname{LMod}_R^1))
\end{align*}
where $D^{\operatorname{cpt}}(\operatorname{Fun}(\Pi(X,P),\operatorname{LMod}_R^1))$ is the full subcategory spanned by the complexes of functors $F_\bullet$ such that $F_\bullet(x)$ is a perfect complex for all objects $x\in \Pi(X,P)$.
\end{corollary}

\section{Change of base category}\label{change of base}

Given that many stacks of interest arise in an algebrogeometric setting, we now provide an explicit description of the \textit{underlying topological stack} $X^{\operatorname{top}}$ of an algebraic stack $X$. This is based on an analogous description of the analytification $X^{\operatorname{an}}$ made by Hinich--Vaintrob in \cite[Proposition 4.2.2]{HinichVaintrob}. General change of geometric context is treated more formally in for example \cite[\S 2.5]{PortaYu16}. Note that Nocera treats the formalism of stratified algebraic stacks and \textit{stratified} analytification in \cite[Appendix B]{Nocera}.

By an \textit{algebraic $\infty$-stack}, we mean a sheaf $X\in \operatorname{Shv}(\operatorname{Aff}_{\C}^{ft})$, where $\operatorname{Aff}_{\C}^{ft}$ denotes the category of affine schemes of finite type over $\operatorname{Spec}(\C)$ equipped with the étale topology. A common change of base procedure is given by transferring the \v{C}ech nerve of an atlas into a different base category. In the case at hand, we apply the functor
\begin{align*}
{\scriptstyle(-)}^{\operatorname{top}}\colon \operatorname{Aff}_{\C}^{ft}\rightarrow \mathfrak{T},\quad U\rightarrow U^{\operatorname{top}},
\end{align*}
sending an affine scheme to the topological space underlying its analytification (see \cite[Exposé XII]{SGA1}). Note that the analytification of an affine scheme of finite type over $\C$ is indeed a second countable locally compact Hausdorff space as it embeds into $\C^n$ with the usual Euclidean topology.

The advantage of the definition that we will give below (\Cref{underlying topological stack}) is that it does not require a choice of atlas. We start on presheaf categories and begin with a simple preliminary observation.

\begin{remark}\label{extending to schemes of non-finite type}
A presheaf $X\in \mathcal{P}(\operatorname{Aff}_{\C}^{ft})$ extends canonically to a presheaf on the category $\operatorname{Aff}_{\C}$ of affine schemes over $\C$ (not necessarily of finite type) via the left adjoint to restriction along the inclusion $\iota\colon\operatorname{Aff}_{\C}^{ft}\hookrightarrow \operatorname{Aff}_{\C}$:
\begin{align*}
\overline{\scriptstyle(-)}\colon \mathcal{P}(\operatorname{Aff}_{\C}^{ft})\rightarrow \mathcal{P}(\operatorname{Aff}_{\C})
\end{align*}
This is explicitly given by left Kan extension along $\iota^{op}$ and it restricts to the inclusion $\iota$ along the Yoneda embeddings (\cite[Corollary F]{HaugsengHebestreitLinskensNuiten}). In particular, for any $U\in \operatorname{Aff}_{\C}^{ft}$, the presheaf $\overline{U}$ is just the representable $U$ but now over the category $\operatorname{Aff}_{\C}$.
\end{remark}

We will need the following lemma.

\begin{lemma}\label{canonical extension is left exact}
The functor $\overline{\scriptstyle(-)}$ is left exact.
\end{lemma}
\begin{proof}
This follows by the pointwise formula for left Kan extension, the fact that finite limits and filtered colimits commute and finally the observation that $\operatorname{Aff}_{\C}^{\operatorname{op}}$ is the $\operatorname{Ind}$-completion of $(\operatorname{Aff}_{\C}^{ft})^{\operatorname{op}}$, $\operatorname{Aff}_{\C}^{\operatorname{op}}\simeq \operatorname{Ind}((\operatorname{Aff}_{\C}^{ft})^{\operatorname{op}})$; in other words, any affine scheme is a filtered limit of affine schemes of finite type (\cite[\S 5.3.5]{LurieHTT}).
\end{proof}

Allow us to make a small observation before we continue; it will not be needed in what's to come, but it's a nice little fact.

\begin{proposition}
If $X\in \operatorname{Shv}(\operatorname{Aff}_{\C}^{ft})\subseteq \mathcal{P}(\operatorname{Aff}_{\C}^{ft})$ is an étale sheaf taking values in $d$-truncated spaces, $\tau_{\leq d}\mathcal{S}$, for some $d<\infty$, then the left Kan extension $\overline{X}\in \mathcal{P}(\operatorname{Aff}_{\C})$ is an étale sheaf on $\operatorname{Aff}_{\C}$.
\end{proposition}
\begin{proof}
Let $\{A\rightarrow B_i\}_{i\in I}$ be the generator of an étale cover of $\operatorname{Spec}(A)\in \operatorname{Aff}_{\C}$, i.e. a finite family of étale morphisms of $\C$-algebras such that the map $A\rightarrow \prod_{i\in I}B_i$ is surjective on $\operatorname{Spec}$. There exist
\begin{itemize}
\item a finite type $\C$-algebra $A'$ together with a map $A'\rightarrow A$,
\item and a family of étale morphisms $\{A'\rightarrow B_i'\}_{i\in I}$,
\end{itemize}
such that $A'\rightarrow \prod_{i\in I} B_i'$ is surjective on $\operatorname{Spec}$ and $B_i\cong A\otimes_{A'}B_i'$ for all $i\in I$. Indeed, each map $A\rightarrow B_i$ is of finite presentation, so we can lift it to any finite type subring $A'\hookrightarrow A$. It remains to be seen that we can ensure (after possibly enlarging $A'$) that the resulting morphism $A'\rightarrow B'=\prod_{i\in I} B_i'$ is faithfully flat and unramified.

Recall that a morphism is unramified if and only if the Kähler differentials vanish; since the morphism $A\rightarrow B =\prod B_i$ is of finite presentation, the Kähler differentials $\Omega_{B/A}$ form a finitely presented module over $B$ generated by the differentials of the generators of $B$ with relations given by the differentials of the relations of $B$. As this is all finite over $A$, it follows that if $\Omega_{B/A}=0$, then we can choose the finite type subring $A'\subseteq A$ large enough for $\Omega_{B'/A'}$ to vanish as well. Faithful flatness is ensured by \cite[Theorem 8.10.5 and Corollary 11.2.6.1]{EGAIV3}.

Now, let $X\in \operatorname{Shv}(\operatorname{Aff}_{\C}^{ft})$ and assume that $X$ takes values in $\tau_{\leq d}\mathcal{S}$. Given $A'\rightarrow A$ and $\{A'\rightarrow B_i'\}_{i\in I}$ as above, let us prove that the descent diagram of $\overline{X}$ over $A$ is a limit diagram. We have
\begin{align*}
A\cong \varinjlim_{A'\rightarrow A''\rightarrow A} A''
\end{align*}
where the filtered colimit is indexed over factorisations of the map $A'\rightarrow A$ through finite type $\C$-algebras $A''$ (using that $\operatorname{Aff}_{A'}^{\operatorname{op}}$ is the $\operatorname{Ind}$-completion of $(\operatorname{Aff}_{A'}^{ft})^{\operatorname{op}}$). Hence, for all $i\in I$, we have
\begin{align*}
B_i\cong \varinjlim_{A'\rightarrow A''\rightarrow A} A''\otimes_{A'}B_i'
\end{align*}
and likewise for all iterated tensor products of the $B_i$ over $A$. It follows that the descent diagram over $A$ is a filtered colimit of descent diagrams over $A''$. The latter are limit diagrams as we've assumed $X$ to be a sheaf, and since $X$ takes values in $d$-truncated spaces, these limits are finite (\cite[Corollary 6.13]{Du}). The desired claim follows since filtered colimits commute with finite limits.
\end{proof}

We use the canonical extension to define the underlying topological presheaf of an algebraic presheaf.

\begin{definition}
Let $\widehat{\scriptstyle(-)}^{pre}\colon \mathcal{P}(\operatorname{Aff}_{\C}^{ft})\rightarrow \mathcal{P}(\mathfrak{T})$ be the functor that sends a presheaf $X$ over $\operatorname{Aff}_{\C}^{ft}$ to the following presheaf over $\mathfrak{T}$,
\begin{align*}
S \mapsto \overline{X}(\operatorname{Spec}(C(S,\C))),
\end{align*}
sending a topological space $S\in \mathfrak{T}$ to the value of the canonical extension of $X$ on the ring of complex continuous functions on $S$. A continuous map $S\rightarrow T$ is sent to the map induced by precomposition.
\end{definition}

\begin{remark}
We simply need the canonical extension as an intermediary construction as the ring $C(S,\C)$ is not in general of finite type. Analytification beyond the finite type setting has been considered in \cite[\S 6.1]{PortaYu20} and \cite[\S 3]{HolsteinPorta}.
\end{remark}

The following lemma tells us what happens on representables.

\begin{lemma}\label{underlying topological stack on representables}
For any $U\in \operatorname{Aff}_{\C}^{ft}$, there is a functorial equivalence $\widehat{U}^{pre}\xrightarrow{\sim} U^{\operatorname{top}}$, where we identify $U$ and $U^{\operatorname{top}}$ with their associated representables.
\end{lemma}
\begin{proof}
Note first of all that for $U=\operatorname{Spec}(A)$, the topological space $U^{\operatorname{top}}$ can also be defined by equipping the set of complex points $U(\C)=\operatorname{Hom}_{\C}(A,\C)$ with the initial topology with respect to the maps
\begin{align*}
U(\C)\rightarrow \C, \quad f\mapsto f(a),\quad a\in A
\end{align*}
(see for example \cite[\S 2.2 including the footnotes]{Payne}). It follows that for $U=\operatorname{Spec}(A)$ and any $S\in \mathfrak{T}$, both
\begin{align*}
\widehat{U}^{pre}(S)=\operatorname{Map}_{\operatorname{CRing}}(A, C(S,\C))\quad\text{and}\quad U^{\operatorname{top}}(S)=\operatorname{Map}_{\mathfrak{T}}(S,\operatorname{Hom}_{\C}(A,\C)) 
\end{align*}
identify with the set of maps $f\colon A\times S\rightarrow \C$ satisfying that
\begin{itemize}
\item for all $a\in A$, the map $f(a,-)\colon S\rightarrow \C$ is continuous;
\item for all $s\in S$, the map $f(-,s)\colon A\rightarrow \C$ is a homomorphism of $\C$-algebras.\qedhere
\end{itemize}
\end{proof}

We have the following more or less immediate corollary.

\begin{corollary}
The functor $\widehat{\scriptstyle(-)}^{pre}$ is left adjoint to restriction along ${\scriptstyle(-)}^{\operatorname{top}}$.
\end{corollary}
\begin{proof}
In view of the lemma above, the diagram below is homotopy-commutative where the horizontal maps are the Yoneda embeddings.
\begin{center}
\begin{tikzpicture}
\matrix (m) [matrix of math nodes,row sep=2em,column sep=2em]
  {
\operatorname{Aff}_{\C}^{ft} & \mathcal{P}(\operatorname{Aff}_{\C}^{ft}) \\
\mathfrak{T} & \mathcal{P}(\mathfrak{T}) \\
  };
  \path[-stealth]
(m-1-1) edge node[left]{${\scriptstyle(-)}^{\operatorname{top}}$} (m-2-1) edge (m-1-2)
(m-1-2) edge node[right]{$\widehat{\scriptstyle(-)}^{pre}$} (m-2-2) 
(m-2-1) edge (m-2-2)
;
\end{tikzpicture}
\end{center}

The functor $\widehat{\scriptstyle(-)}^{pre}$ is defined as a composite of two left adjoints:
\begin{align*}
\widehat{\scriptstyle(-)}^{pre}\colon \mathcal{P}(\operatorname{Aff}_{\C}^{ft})\xrightarrow{\overline{\scriptstyle(-)}}\mathcal{P}(\operatorname{Aff}_{\C})\xrightarrow{c^*} \mathcal{P}(\mathfrak{T})
\end{align*}
where $c^*$ is restriction along the functor $c\colon \mathfrak{T}\rightarrow \operatorname{Aff}_{\C}$, $S\mapsto C(S,\C)$. It follows that we can identify $\widehat{\scriptstyle(-)}^{pre}$ as the unique colimit preserving functor restricting to ${\scriptstyle(-)}^{\operatorname{top}}$ along the Yoneda embeddings; equivalently, as the left adjoint to the restriction along ${\scriptstyle(-)}^{\operatorname{top}}$ (\cite[Corollary F]{HaugsengHebestreitLinskensNuiten}).
\end{proof}

Now we move onto $\infty$-stacks.

\begin{definition}\label{underlying topological stack}
Consider the composite
\begin{align*}
\widehat{\scriptstyle(-)}:=L\circ \widehat{\scriptstyle(-)}^{pre}\circ \iota \colon \operatorname{Shv}(\operatorname{Aff}_{\C}^{ft})\rightarrow \operatorname{Shv}(\mathfrak{T})
\end{align*}
where $\iota$ is the inclusion of the sheaf $\infty$-category and $L$ is the sheafification functor. We call $\widehat{X}\in \operatorname{Shv}(\mathfrak{T})$ the \textit{underlying topological $\infty$-stack} of  $X\in \operatorname{Shv}(\operatorname{Aff}_{\C}^{ft})$.

Explicitly, for $X\in \operatorname{Shv}(\operatorname{Aff}_{\C}^{ft})$, the underlying topological $\infty$-stack $\widehat{X}\in \operatorname{Shv}(\mathfrak{T})$ is the sheafification of the presheaf
\begin{align*}
S\mapsto \overline{X}(C(S,\C))
\end{align*}
evaluating the canonical extension of $X$ on the ring of complex continuous functions on $S$.
\end{definition}

\begin{lemma}\label{taking underlying topological stack is left adjoint}
The functor $\widehat{\scriptstyle(-)}$ is left adjoint to restriction along ${\scriptstyle(-)}^{\operatorname{top}}$ (at the level of sheaves).
\end{lemma}
\begin{proof}
Note that ${\scriptstyle(-)}^{\operatorname{top}}\colon \operatorname{Aff}_{\C}^{ft}\rightarrow \mathfrak{T}$ sends surjective étale morphisms to surjective local homeomorphisms (\cite[Exposé XII, Propositions 3.1 and 3.2]{SGA1}), and that it also preserves fibre products (\cite[Exposé XII, \S 1.2]{SGA1}). It follows that ${\scriptstyle(-)}^{\operatorname{top}}$ preserves covering families and consequently that $\widehat{\scriptstyle(-)}^{pre}$ preserves local equivalences. Hence, restriction along ${\scriptstyle(-)}^{\operatorname{top}}$ (at the level of presheaves) preserves sheaves and thus restricts to a morphism
\begin{align*}
t^*\colon \operatorname{Shv}(\mathfrak{T})\rightarrow\operatorname{Shv}(\operatorname{Aff}_{\C}^{ft}),
\end{align*}
which, by uniqueness of adjoints, must be right adjoint to $\widehat{\scriptstyle(-)}=L\circ \widehat{\scriptstyle(-)}^{pre}\circ \iota$.
\end{proof}

We now show that the change of base defined in \Cref{underlying topological stack} recovers the more common change of base procedure outlined at the start of this section, that is, transferring the \v{C}ech nerve of an algebraic stack into the topological setting via the functor ${\scriptstyle(-)}^{\operatorname{top}}\colon \operatorname{Aff}_{\C}^{ft}\rightarrow \mathfrak{T}$ and considering the resulting topological $\infty$-stack. 

\begin{proposition}
Let $X\in \operatorname{Shv}(\operatorname{Aff}_{\C}^{ft})$ be an algebraic $\infty$-stack and $f\colon U\rightarrow X$ a representable effective epimorphism from an affine scheme $U\in \operatorname{Aff}_{\C}^{ft}$. Then the induced map $\hat{f}\colon \widehat{U}\rightarrow \widehat{X}$ is a representable effective epimorphism whose \v{C}ech nerve identifies with the underlying topological \v{C}ech nerve of $f$, $\check{C}(\hat{f})\cong \check{C}(f)^{\operatorname{top}}$. In particular, $\widehat{X}$ identifies with the colimit of $\check{C}(f)^{\operatorname{top}}$ in $\operatorname{Shv}(\mathfrak{T})$:
\begin{align*}
\widehat{X}\xleftarrow{\ \simeq \ } \varinjlim_{[n]\in \Delta^{op}} \left(U^{\resizebox{5pt}{!}{$\displaystyle\mathop{\times}_X$}n} \right)^{\operatorname{top}}.
\end{align*}
If $f\colon U\rightarrow X$ is étale, so is $\hat{f}\colon \widehat{U}\rightarrow \widehat{X}$. In other words, $\widehat{\scriptstyle(-)}$ sends Deligne--Mumford stacks to étale topological $\infty$-stacks.
\end{proposition}
\begin{proof}
Recall that $\widehat{\scriptstyle(-)}=L\circ c^* \circ \overline{\scriptstyle(-)}\circ \iota$ and note that $\iota$ and $c^*$ are right adjoints and $L$ and $\overline{\scriptstyle(-)}$ are left exact (\Cref{canonical extension is left exact}). It follows that $\widehat{\scriptstyle(-)}$ is a left exact left adjoint (\Cref{taking underlying topological stack is left adjoint}) and as such it preserves \v{C}ech nerves and effective epimorphisms. The claim then follows directly from \Cref{underlying topological stack on representables}.
\end{proof}

In view of this result, we also write $\widehat{X}=X^{\operatorname{top}}$.

\begin{remark}\label{not Artin stacks}
We note that our focus on étale topological $\infty$-stacks in the previous sections mean that, as written, the results of this note apply mostly to Deligne--Mumford stacks, not Artin stacks, i.e. algebraic stacks admitting a \textit{smooth} atlas.
\end{remark}

We make a small additional observation about coarse moduli spaces. For a scheme $U\in \operatorname{Sch}_{\C}^{ft}$ of finite type over $\C$, we also denote by $U$ the algebraic stack
``represented'' by $U$, i.e. sending $V\in \operatorname{Aff}_{\C}^{ft}$ to $\operatorname{Map}_{\operatorname{Sch}_{\C}^{ft}}(V,U)$.

\begin{definition}
Let $X\in \operatorname{Shv}(\operatorname{Aff}_{\C}^{ft})$ be an algebraic stack. A \textit{coarse moduli space} for $X$ is a scheme $U\in \operatorname{Sch}_{\C}^{ft}$ together with a morphism $\tau\colon X\rightarrow U$ such that
\begin{enumerate}
\item $\tau$ induces a bijection on geometric points
\begin{align*}
\pi_0 X(\operatorname{Spec}(\C))\xrightarrow{\ \cong \ }\operatorname{Map}_{\operatorname{Sch}_{\C}^{ft}}(\operatorname{Spec}(\C),U).
\end{align*}
\item for any scheme $U'\in \operatorname{Sch}_{\C}^{ft}$ and any morphism $f\colon X\rightarrow U'$, there is an essentially unique factorisation of $f$ through $\tau$:
\begin{center}
\begin{tikzpicture}
\matrix (m) [matrix of math nodes,row sep=2em,column sep=2em]
  {
X & U \\
  & U' \\
  };
  \path[-stealth]
(m-1-1) edge node[above]{$\tau$} (m-1-2)
(m-1-1) edge node[below left]{$f$} (m-2-2)
;
\path[-stealth, dashed]
(m-1-2) edge (m-2-2)
;
\end{tikzpicture}
\end{center}
\end{enumerate} 
We also say that the morphism $\tau$ exhibits $U$ a \textit{coarse moduli space} for $X$.
\end{definition}

\begin{proposition}
Let $X\in \operatorname{Shv}(\operatorname{Aff}^{ft}_{\C})$ be an algebraic stack and suppose $\tau \colon X\rightarrow U$ exhibits $U\in \operatorname{Sch}_{\C}^{ft}$ as a coarse moduli space for $X$. Then the underlying morphism of topological stacks $X^{\operatorname{top}}\rightarrow U^{\operatorname{top}}$ exhibits the topological space $U^{\operatorname{top}}$ as a coarse moduli space for $X^{\operatorname{top}}$ (see \Cref{coarse moduli space}).
\end{proposition}
\begin{proof}
By definition, the map $X^{\operatorname{top}}(\ast)\rightarrow U^{\operatorname{top}}(\ast)$ identifies with the map $X(\operatorname{Spec}(\C))\rightarrow U(\operatorname{Spec}(\C))$, so item 1 of \Cref{coarse moduli space} is immediate. Denoting by $t^*$ the restriction along ${\scriptstyle(-)}^{\operatorname{top}}\colon \operatorname{Aff}_{\C}^{ft}\rightarrow  \mathfrak{T}$, the adjunction ${\scriptstyle(-)}^{\operatorname{top}}\dashv t^*$ implies that we have a homotopy commutative square as below for any topological space $S$,
\begin{center}
\begin{tikzpicture}
\matrix (m) [matrix of math nodes,row sep=2em,column sep=2em]
  {
\operatorname{Map}_{\operatorname{Shv}(\mathfrak{T})}(X^{\operatorname{top}},S) & \operatorname{Map}_{\operatorname{Shv}(\operatorname{Aff}_{\C}^{ft})}(X,t^*S) \\
\operatorname{Map}_{\operatorname{Shv}(\mathfrak{T})}(U^{\operatorname{top}},S) & \operatorname{Map}_{\operatorname{Shv}(\operatorname{Aff}_{\C}^{ft})}(U,t^*S) \\
  };
  \path[-stealth]
(m-1-1) edge node[above]{$\sim$} (m-1-2)
(m-1-1) edge node[left]{$(\tau^{\operatorname{top}})^*$} (m-2-1)
(m-1-2) edge node[right]{$\tau^*$} (m-2-2)
(m-2-1) edge node[below]{$\sim$} (m-2-2)
;
\end{tikzpicture}
\end{center}
Thus essentially unique factorisation of morphisms $X\rightarrow t^*S$ through $\tau$ implies essentially unique factorisation of morphisms $X^{\operatorname{top}}\rightarrow S$ through $\tau^{\operatorname{top}}$.
\end{proof}

\begin{remark}
We recall that for any $U\in \operatorname{Sch}_{\C}^{ft}$, the underlying topological space $U^{\operatorname{top}}$ is locally compact as it admits a cover by subspaces of $\C^n$. Note moreover that
\begin{itemize}
\item if $U$ is separated, then $U^{\operatorname{top}}$ is Hausdorff (\cite[Exposé XII, Proposition 3.1]{SGA1});
\item if $U$ is complete, then $U^{\operatorname{top}}$ is compact (\cite[Exposé XII, Proposition 3.2]{SGA1}).\qedhere
\end{itemize}
\end{remark}


\end{document}